\documentclass[11pt,reqno]{amsart}
\usepackage{amsfonts,amssymb,amsmath,amsthm}
\usepackage[margin=1in]{geometry}
\usepackage[hidelinks]{hyperref}
\usepackage{cleveref}
\usepackage{enumitem}
\usepackage{microtype}
\usepackage{mathtools}
\usepackage{color}
\usepackage{tikz-cd}
\usepackage{comment}
\usepackage{cleveref} \usepackage{enumitem} \usepackage{microtype} \usepackage{mathtools} \usepackage{color} \usepackage{tikz-cd} \usepackage{comment} \usepackage{graphicx} 
\usepackage[numbers,sort&compress]{natbib}

\def\p#1{\left( #1 \right)}
\def\Z{\mathbb{Z}}
\def\Q{\mathbb{Q}}
\def\F{\mathbb{F}}

\def\R{\mathbb{R}}
\def\C{\mathbb{C}}
\def\fp{\mathfrak{p}}

\def\SL{\operatorname{SL}}

\def\tr{\operatorname{tr}}

\def\disc{\operatorname{disc}}
\def\Num{\operatorname{Num}}
\def\Den{\operatorname{Den}}

\def\End{\operatorname{End}}

\def\kronecker#1#2{\p{\frac{#1}{#2}}}

\theoremstyle{plain}
\newtheorem{theorem}{Theorem}[section]
\newtheorem{corollary}{Corollary}[section]
\newtheorem{lemma}{Lemma}[section]
\newtheorem{proposition}{Proposition}[section]

\theoremstyle{definition}
\newtheorem{definition}[theorem]{Definition}

\newtheorem{conjecture}[]{Conjecture}

\theoremstyle{remark}
\newtheorem{remark}[]{Remark}

\title[Prescribed CM in Reductions of Elliptic Curves]{Prescribed Complex Multiplication in Reductions of Elliptic Curves}
\author[T. Wang]{Tian Wang}
\address{
Tian Wang, Max Planck Institute for Mathematics, Vivatsgasse 7, 53111, Bonn,  Germany
}
\email{tianwmath@gmail.com}

\begin{document}

\begin{abstract}
We study primes at which the reduction of an elliptic curve admits complex multiplication by an order in a prescribed imaginary quadratic field, a problem closely related to the Lang--Trotter conjecture for Frobenius fields. For certain non-CM elliptic curves $E/\mathbb{Q}$ and imaginary quadratic fields $K$, we show that either the number of primes $p\leq x$ at which $E$ has Frobenius field $K$ satisfies a logarithmic-type lower bound in $x$, or there are infinitely many supersingular primes at which the reduction of $E$ has CM by $K$.  Assuming a $p$-adic approximation conjecture for singular moduli, we further obtain a quantitative lower bound in the supersingular case. Our approach extends Elkies’s method for proving the infinitude of supersingular primes, combining a combinatorial argument with refined local estimates for singular moduli and Deuring’s theory of complex multiplication. Finally, we use arithmetic intersection theory on modular curves to establish power-saving upper bounds for a restricted counting function.
\end{abstract}
\maketitle

\section{Introduction}

Let $E/\Q$ be an elliptic curve of conductor $N_E$. Throughout this paper, we assume that $E$ is non-CM, i.e., the geometric ring of endomorphisms $\End_{\overline{\Q}}(E)$ of $E$ is isomorphic to $\Z$.  For each prime $p\nmid N_E$ of good reduction, let $E_p$ denote the reduction of $E$ modulo $p$,  and 
\[
a_p(E):=p+1-\#E_p(\F_p)
\]
be the Frobenius trace  of $E$ at $p$.   The Frobenius morphism $x\mapsto x^p$ induces an endomorphism $\pi_p(E_p)$ of $E_p$, and the characteristic polynomial of $\pi_p(E_p)$ is 
 \[
P_{E, p}(X):=X^2-a_p(E)X+p \in \Z[X].
\]
The roots of $P_{E, p}(X)$ are $\pi_p(E_p)$ and its complex conjugate  $\overline{\pi}_p(E_p)$, both with absolute value $\sqrt{p}$. 
In particular, we have the Hasse--Weil bound  $|a_p(E)|\leq 2\sqrt{p}$, and hence the Frobenius field $\Q(\pi_p(E_p))$ of $E$ at $p\nmid N_E$ is always an imaginary quadratic field. 

  In 1976, Lang and Trotter \cite[p.2]{MR568299} studied the distribution of primes for which $a_p(E)$ equals a fixed integer, and the distribution of primes for which $\Q(\pi_p(E_p))$ equals a fixed imaginary quadratic field.  
  More precisely, let $t$ be a fixed integer and $\Delta\geq 1$  a squarefree integer. Denote  
\[
\pi_E(x, t):=\#\{p\leq x: p\nmid N_E,  a_p(E)=t\},
\]
and 
\begin{equation}\label{eq:original}
\pi_E(x, \Q(\sqrt{-\Delta})):=\#\{p\leq x: p\nmid N_E,  \Q(\pi_p(E_p))=\Q(\sqrt{-\Delta})\}.
\end{equation}
\begin{conjecture}[Lang--Trotter conjecture]\label{conj:L-T} Let $E/\Q$ be a non-CM elliptic curve. Then there exist constants  $C(t, E)\ge 0$ and  $C(\Delta, E)>0$, depending only on \(E\) and \(t\), and on \(E\) and \(\Delta\), respectively, such that   
\[
\pi_E(x, t)\sim C(t, E)\frac{x^{\frac{1}{2}}}{\log x}, \quad \pi_E(x, \Q(\sqrt{-\Delta}))\sim C(\Delta, E)\frac{x^{\frac{1}{2}}}{\log x}
\]
as $x\to \infty$.
\end{conjecture}
The main focus of this paper is to get upper and lower bounds related to the counting function  $\pi_E(x, \Q(\sqrt{-\Delta}))$.

\begin{theorem}\label{thm:lower-bound}
   Let $E/\Q$ be a non-CM elliptic curve, and let $\Delta\geq 1$ be a squarefree integer. Assume
   \begin{enumerate}
    \item \label{assump:Delta}
   $\Delta \equiv -1\pmod 8$.
     \item \label{assump:Kodaira}
     $2\mid \Num(j_E)$ and for every odd prime $p\mid N_E$, the Kodaira symbol at $p$ is $\operatorname{I}_0^*$.
\end{enumerate}
     Then, one of the following alternatives holds:
\begin{enumerate}
 \item We have the following lower bound for all sufficiently large $x$:
         \[
       \pi_E(x, \Q(\sqrt{-\Delta}))\gg_{E, \Delta} \frac{\log x}{(\log\log x)^3}. 
         \]
         \item There are infinitely many supersingular primes $p$ for which $\End_{\overline{\F}_p}(E_p)\otimes_\Z \Q $ contains $ \Q(\sqrt{-\Delta})$.
        
     \end{enumerate}
\end{theorem}

Denote by
\begin{equation}\label{eq:cmtype}
\tilde{\pi}_E(x, \Q(\sqrt{-\Delta})):=\#\{p\leq x: p\nmid N_E, \Q(\sqrt{-\Delta}) \subseteq \End_{\overline{\F}_p}(E_p)\otimes_\Z \Q\}.
\end{equation}
This  function counts the primes $p$ for which the reduction $E_p$ has complex multiplication by an order in $\Q(\sqrt{-\Delta})$ (as defined in \Cref{sec:end}). Note that
\[
\tilde{\pi}_E(x, \Q(\sqrt{-\Delta}))\geq \pi_E(x, \Q(\sqrt{-\Delta})). 
\]
\Cref{thm:lower-bound} then implies 
\[
\tilde{\pi}_E(x, \Q(\sqrt{-\Delta}))\to \infty
 \qquad \text{ as $x\to \infty$}.
 \]
Assuming the uniform version of \Cref{conj:super},  we also give a quantitative lower bound  of $\tilde{\pi}_E(x, \Q(\sqrt{-\Delta}))$ in \Cref{thm:quantitative}.   \Cref{conj:super} may be viewed as a $p$-adic approximation for CM points and is closely related in spirit to the Andr\'e--Oort conjecture, as discussed in \Cref{sec:singular}. 

\begin{remark}\label{rem:rm-1}
To remove the assumptions on $\Delta$, $j_E$, and the Kodaira type in \Cref{thm:lower-bound}, it seems necessary to show that for a non-CM elliptic curve \(E/\Q\) and a singular modulus \(j\) of discriminant \(-f^2\Delta^*\) (\(f\in \Z_{\geq 1}\)), the difference \(j_E-j\) is not an \(S\)-unit for \(S=\{2\}\cup\{p:p\mid N_E\}\). To the best of the author’s knowledge, it is not currently known whether results of this type are available in the literature (see, for example, \cite{MR772491, MR3404647, MR4190395, MR4251608, MR4296371, MR4713026} for related discussions).
 \end{remark}

\begin{remark}
Elliptic curves $E/\Q$ in the isogeny class $128.a$ \cite{lmfdb} satisfy the condition \eqref{assump:Kodaira} of \Cref{thm:lower-bound}. 

As a first example,  take $\Delta=7$ and let $E/\Q$ be the elliptic curve with LMFDB label $128.a1$ with $j_E=10976$. Using SageMath, we also compute the maximal $p$-adic valuation of $P_{7f^2}(j_E)$ for   primes  $p\leq 10^{15}$ and $f\leq 100$, where $P_{7f^2}(X)$ denotes the Hilbert class polynomial of discriminant $-7f^2$. The results are given in  \Cref{tabel}.
In this range, the maximal valuations occur at primes $p=31$ and $223$, both of which are inert in $\Q(\sqrt{-7})$ and hence are supersingular primes for $E$. 

As a second example, take $\Delta=15$ and the same elliptic curve $E/\Q$.  
In \Cref{tabel-2}, we present the maximal $p$-adic valuations of  $P_{15f^2}(j_E)$ for primes $p\leq 10^{15}$ and $f\leq 100$. An $\ast$ in the table indicates that the smallest prime divisor of $P_{15f^2}(j_E)$ exceeds $10^{15}$. In this range, the maximal valuations occur at both ordinary and supersingular primes. However, as $f$ grows, the larger valuations are concentrated at the fixed supersingular primes $p=223$ and $479$.

The computations suggest that, as $f$ grows, large valuations are concentrated at these fixed supersingular primes rather than arising from increasingly large primes. This behavior is consistent with the expected dependence on $p$ in \Cref{conj:super}.
\end{remark}


A classical unconditional result concerning the counting functions in  \Cref{conj:L-T} comes from the seminal work of Elkies \cite{MR903384}, who proved that there are infinitely many supersingular primes for every elliptic curve $E/\Q$. This is equivalent to
\[
\pi_E(x,0)\to\infty \qquad \text{as } x\to\infty.
\]
The proof is reminiscent of Euclid's argument for the infinitude of primes and relies crucially on Deuring's theory of complex multiplication together with arithmetic properties of Hilbert class polynomials. Elkies and Murty \cite[Corollary, p.\,21]{MR2636029} also gave the best known quantitative lower bound, $\pi_E(x,0)\gg \log\log x$, conditional on the Generalized Riemann Hypothesis (GRH). Unconditionally, the bound  $\pi_E(x, 0)\gg_\epsilon \frac{\log\log\log x}{(\log\log\log\log x)^{1+\epsilon}}$ for any $\epsilon>0$  is due to Fouvry and Murty \cite{MR1382477}.
However, no lower bound  was known for  $\pi_E(x,\Q(\sqrt{-\Delta}))$, even for a single fixed choice of $E$ and $\Delta$. 

For the related counting function $\tilde{\pi}_E(x,\Q(\sqrt{-\Delta}))$, Charles \cite[Corollary 1.2]{MR3843371} proved that
$
\tilde{\pi}_E(x,\Q(\sqrt{-\Delta}))\to\infty
$ as $x\to\infty$ 
for every elliptic curve $E/\Q$ and every imaginary quadratic field. 
Under our hypotheses, \Cref{thm:lower-bound} refines this infinitude statement by providing a quantitative lower bound in the ordinary alternative. Our proof is different from Charles's geometric approach. We use a constructive argument similar to that of Elkies's, combining new combinatorial argument with refined local estimates for singular moduli. The principal novelty is that these ingredients can provide quantitative lower bounds in either the ordinary or supersingular case, with the supersingular bound conditional on \Cref{conj:super}.


In contrast, upper bounds for the counting functions in \Cref{conj:L-T} have been extensively studied. We briefly recall the best known bounds for
$\pi_E(x,\Q(\sqrt{-\Delta}))$. \footnote{For the counting function $\pi_E(x,t)$, see \cite{MR4586829} and the references therein.}
Under GRH, Serre \cite[p.~191]{MR644559} showed that
$
\pi_E(x,\Q(\sqrt{-\Delta}))$ is bounded above by $O_{E,\Delta} (x^\theta)
$
for some $\theta$ with $1/2\leq\theta<1$. Subsequent refinements
\cite{MR2178556,MR2464027,MR3453123} led to the bound 
$
O_{N_E,\Delta} (x^{4/5}/(\log x)^{3/5}).
$
Unconditionally, Serre \cite[p.~191]{MR644559} obtained the bound 
$
O_{E,\Delta}(x/(\log x)^\gamma)$
for some $\gamma>1$. Following subsequent improvements
\cite[Theorem 1.3]{MR2178556} and \cite[Theorem 1.3(ii)]{MR3453123},
the current best bound, due to Thorner and Zaman
\cite[Theorem 1.5]{MR3848226}, is
$
O_{N_E,\Delta} (x\log\log x/(\log x)^2).
$
To get this result,  they need to prove an improved unconditional effective Chebotarev density theorem and then apply sieve methods.
This suggests that even removing the $\log\log x$ factor unconditionally is a nontrivial problem.

The next result is to bound the number of primes $p$ for which $E_p$ has CM by some order in $\Q(\sqrt{-\Delta})$ with small conductor. 
More explicitly, for any constant $\delta$ satisfying $0<\delta<1/2$, we  give an upper bound
with a power-saving in $x$ for  
\begin{equation}\label{eq:weaker-bound}
    \tilde{\pi}_E(x, \Q(\sqrt{-\Delta}); \delta):=\#\{p\leq x: p\nmid N_E, \mathcal{O}_{f_p^2\Delta^*} \subseteq \End_{\overline{\F}_p}(E_p) \text{ for some } f_p\leq 2p^{\delta}\},
\end{equation}
where $\mathcal{O}_{f_p^2\Delta^*}$ is the order of conductor $f_p$ in  $\mathcal{O}_{\Q(\sqrt{-\Delta})}$. We will see in \Cref{sec:proof-upper} that the number of ordinary primes counted by  $  \tilde{\pi}_E(x, \Q(\sqrt{-\Delta}); \delta)$ is at most 
$\pi_E(x, \Q(\sqrt{-\Delta}))$, and if $\delta=1/2$, then  $\pi_E(x, \Q(\sqrt{-\Delta}))\leq  \tilde{\pi}_E(x, \Q(\sqrt{-\Delta}); \delta)+O(1)$. 

\begin{theorem}\label{thm:upper-bound}
    Let $E/\Q$ be a non-CM elliptic curve and $\Delta>0$ be a squarefree integer. Let $0<\delta<1/2$. Then
     \[
     \tilde{\pi}_E(x, \Q(\sqrt{-\Delta}); \delta)\ll_{E}  \Delta^{1/2}(\log (3\Delta))^2 x^{2\delta}. 
     \]
\end{theorem}


Although \Cref{thm:upper-bound} does not directly provide an upper bound for $\pi_E(x,\Q(\sqrt{-\Delta}))$, the argument introduces a method that seems to be new in comparison with those used in the existing literature. In addition, it is unclear  whether one can obtain an unconditional power-saving upper bound for $\tilde{\pi}_E(x,\Q(\sqrt{-\Delta});\delta)$ using Galois representations,  the Chebotarev density theorem, and sieve methods, as in previous approaches.

We now summarize the main ideas behind the proofs of the two theorems. Denote 
\begin{equation}\label{eq:allprimes}
\mathcal{P}(E):=\{p: \text{ $p\nmid N_E$, }   \Q(\sqrt{-\Delta}) \subseteq \End_{\overline{\F}_p}(E_p)\otimes_\Z \Q\}.
\end{equation}
We prove that the set $\mathcal{P}(E)$ is infinite by showing that either the set of supersingular primes or the set of ordinary primes in $\mathcal{P}(E)$ is infinite, using a combinatorial trick. The two cases require different arguments. In the supersingular case, we use effective $p$-adic and archimedean equidistribution results for CM points. In the ordinary case, we apply an argument reminiscent of Euclid's proof of the infinitude of primes. Before presenting the details, we explain the role of the assumptions on $\Delta$, $j_E$, and the Kodaira types. The assumptions on $\Delta$ and $j_E$ ensure that $\mathcal{P}(E)$ contains an odd prime $p$, while the assumptions on the Kodaira types allow us to choose a quadratic twist for which $p$ is a prime of good reduction for the twist. 

  We begin by studying prime divisors of Hilbert class polynomial values. In \Cref{lem:notpm1}, we establish a lower bound for $\lvert P_{f^2D}(j_E)\rvert$ as $f\to\infty$, ensuring that these values admit prime divisors for all sufficiently large $f$. Combining this estimate with the  hypotheses on $E$, \Cref{prop:nonempty} provides a density dichotomy for the auxiliary primes, according to whether all relevant class-polynomial values have only supersingular prime divisors or at least one has an ordinary prime divisor. This combinatorial alternative explains for the dichotomy in \Cref{thm:lower-bound}.


Suppose first that the supersingular alternative in \Cref{prop:nonempty} holds for a set of auxiliary primes $\ell$ of positive upper density. By \Cref{lem:notpm1}, we have a lower bound for $\log|\Num(P_{L_\ell^2\Delta^*}(j_E))|$ for an auxiliary integer $L_\ell$. On the other hand, we could choose the conductors $L_\ell$ so that \Cref{lem:ss-conj} applies at each relevant supersingular prime. This lemma establishes a special case of \Cref{conj:super} at odd supersingular primes and bounds their local contributions to $\Num(P_{L_\ell^2\Delta^*}(j_E))$.  Summing these contributions yields an upper bound incompatible with the lower bound as $\ell$ tends to infinity.   We therefore conclude that $\mathcal{P}(E)$ contains infinitely many supersingular primes (\Cref{prop:supersingular-infinite}).

Suppose instead that the ordinary alternative in \Cref{prop:nonempty} holds for infinitely many auxiliary primes $\ell$. For each such $\ell$, we choose an auxiliary integer $L_\ell$ for which $\Num(P_{L_\ell^2\Delta^*}(j_E))$ has an ordinary prime divisor. Instead of estimating the local contribution at each ordinary prime, we adapt Elkies's argument and use an iterative construction to produce new primes, thereby proving infinitude. Working with a suitable quadratic twist $E'$ of $E$, we show that, given any finite set of ordinary primes $\Sigma_0(E')\subseteq \mathcal{P}(E')$, we can choose $L_\ell$ so that no prime in $\Sigma_0(E')$ divides $\Num(P_{L_\ell^2\Delta^*}(j_E))$. An ordinary prime divisor of this value therefore lies outside $\Sigma_0(E')$. These results are proved in  \Cref{prop:admissible} and \Cref{lem:new-prime}. Iterating the construction yields infinitely many ordinary primes in $\mathcal{P}(E')$, and the quadratic twist argument gives the same conclusion for $\mathcal{P}(E)$ (\Cref{prop:ord-set-infinite}).


To prove the quantitative lower bound, we again treat the supersingular and ordinary cases separately. In the supersingular case, we assume the uniform version of \Cref{conj:super}, which allows us to bound the $p$-adic contribution for supersingular primes $p$ dividing $\Num(P_{L_\ell^2\Delta^*}(j_E))$. This leads to the lower bound in \Cref{thm:lower-bound-ss}. In the ordinary case, we bound the size of the auxiliary integer $L$ and show that distinct choices of three admissible integers $L$ give rise to at least two new distinct primes. Counting these primes then yields the lower bound in \Cref{thm:lower-bound-ord}. Together, these two results also give a lower bound for $\tilde{\pi}_E(x,\Q(\sqrt{-\Delta}))$ (\Cref{thm:quantitative}).

The proof of \Cref{thm:upper-bound} is of a different nature. We use equidistribution and Faltings heights of CM points together with arithmetic intersection theory on the coarse moduli scheme of generalized elliptic curves $X(1)$ over $\operatorname{Spec}\mathbb{Z}$.   Roughly
speaking, we apply \Cref{prop:Charles} to the global
intersections between the horizontal divisor associated with
$E$ and the $f$-Hecke translates of the CM divisor
$y_{\Delta^*}$ defined in \eqref{eq:heggner}, and then
average over $1\leq f\leq 2x^\delta$.
For all but finitely many primes $p$ counted by $\tilde{\pi}_E(x, \Q(\sqrt{-\Delta}); \delta)$, there exists an integer $1\leq f(p)\leq 2x^\delta$ such that there is a nontrivial contribution to the local intersection in the fiber at $p$.  By giving upper bounds for the archimedean and
global intersection numbers using $q$-expansions of modular
forms, and lower
bounds for the non-archimedean intersection numbers, we show
that the number of primes dividing $\Num(P_{f^2\Delta^*}(j_E))$  cannot be too large. This leads to the desired upper bound. 
 
\subsection*{Outline of the paper}
We now outline the structure of the paper. In \Cref{sec:preliminaries}, we introduce notation and review the necessary background on endomorphism rings of elliptic curves and arithmetic intersection theory on modular curves. In \Cref{sec:singular}, we introduce \Cref{conj:super} and prove several lemmas concerning local properties of  singular moduli.

In \Cref{sec:Theorem2}, we prove both qualitative and quantitative results for primes in $\mathcal{P}(E)$. In particular, in \Cref{sec:newprimes}, we establish the existence of infinitely many relevant supersingular and ordinary primes in \Cref{prop:supersingular-infinite} and \Cref{prop:ord-set-infinite}, respectively. In \Cref{sec:quantitative-lower}, we derive the corresponding quantitative lower bounds for supersingular and ordinary primes in \Cref{thm:lower-bound-ss} and \Cref{thm:lower-bound-ord}, respectively. At the end of \Cref{sec:quantitative-lower}, we prove \Cref{thm:lower-bound}.

  In the last section, we first
establish a trivial upper bound in \Cref{sec:easy-upper-bound} and then  prove \Cref{thm:upper-bound} in \Cref{sec:proof-thm-upper}.



We conclude this section by listing several questions and observations that may be of interest for future research. The reader may safely skip the remainder of this section without affecting their understanding of the rest of the paper.

\begin{remark}
Our results are stated over the rational number field $\Q$, but the tools used in their proofs have also been developed over general number fields. It is therefore plausible that our results extend to this broader setting, though establishing such an extension may require additional work. This possibility is one reason we formulate our assumptions in terms of the $j$-invariant and the reduction types of $E$. We leave this direction for future research.
\end{remark}
\begin{remark}
   After proving the lower bound for $\pi_E(x,0)$, Elkies \cite{MR1144318} also established the unconditional upper bound $
\pi_E(x, 0)\ll x^{\frac{3}{4}}.$
A key observation in his proof (with Murty) is that the product of supersingular primes $p\leq x$ divides the product of $\Num(P_D(j_E))$ over discriminants $D\ll x^\theta$ for some $\theta$. This observation leads to the upper bound
$O\left(x^{3\theta/2}/\log x\right)$. By applying a result of Kaneko \cite{MR1040429}, one can take $\theta=1/2$, which gives the stated bound for $\pi_E(x,0)$.

The proof of \Cref{thm:upper-bound} builds on this idea, and we briefly explain why it leads to a different result. Instead of considering $\Num(P_D(j_E))$ for all discriminants $D\ll x^\theta$, we consider $\Num(P_{f^2\Delta^*}(j_E))$ for all conductors $f\leq 2x^{1/2}$. As noted above, the choice $\theta=1/2$ in Elkies' argument relies crucially on the fact that the primes counted by $\pi_E(x,0)$ are supersingular. By contrast, apart from finitely many exceptions, most of the primes counted by $\pi_E(x,\Q(\sqrt{-\Delta}))$ are ordinary. Achieving $\theta=1/2$ in our setting would essentially require $f\ll x^{1/4}$ (equivalently, $\delta\le 1/4$), which is not expected to hold in general.

\end{remark}

 \begin{remark}
     We note that the use of arithmetic intersection theory to  study   reductions of elliptic curves and abelian varieties is well established. See, for example, the survey article \cite{MR4875529} and the references therein. The methods  have been successfully applied to prove the existence of infinitely many primes with prescribed reduction properties for certain abelian varieties. To the best of the author’s knowledge, they have not previously been used to obtain upper bounds of the type considered in the present work. 

     However, compared with the ``trivial'' bound in \Cref{lem:trivial}, the improvement provided by \Cref{thm:upper-bound} is relatively modest.  A similar phenomenon also appears when applying the unconditional effective Chebotarev density theorem, where obtaining a substantial improvement beyond the logarithmic savings is difficult. It would therefore be interesting to determine whether significantly stronger upper bounds can be obtained by using arithmetic intersection theory or other methods. For instance, a better  result in either the lower bound for \eqref{eq:non-arch-lower} or the upper bound for $\omega\left(\Num(P_{f^2\Delta^*}(j_E)\right)$ appearing in  \Cref{lem:trivial} would potentially give the bound $\pi_E(x, \Q(\sqrt{-\Delta}))\ll_{E, \Delta, \epsilon} x^{1/2+\epsilon}$ for any $\epsilon>0$. 
   \end{remark}

\subsection*{List of notation and conventions}

\begin{enumerate}
    \item Let $D>0$ be an integer such that $D\equiv 0, 3\pmod 4$, and let  $\mathcal{O}_D:=\Z\left[\frac{D+\sqrt{-D}}{2}\right]$ be the order of discriminant $-D$ in $K:=\Q(\sqrt{-D})$. We say $-D$ is a (fundamental) discriminant if $\mathcal{O}_D$ is an order (the maximal order) of $K$. We denote
by $h(D)$ and $P_D(X)$ the class number and the Hilbert
class polynomial of $\mathcal{O}_D$, respectively.
    \item For a positive integer $n$, let $\omega(n)$ denote the number of distinct prime divisors of $n$. For $x\in\Q$, let $\Num(x)$ and $\Den(x)$ denote the numerator and denominator of $x$ when written in lowest terms, respectively. For $r\in\R$, we denote $\exp(r):=e^r$. 
    \item Let $\alpha\in\overline{\Q}$ be an algebraic integer with minimal polynomial 
    \[ P(X)=(X-\alpha_1)\cdots(X-\alpha_n), 
    \] 
    where $\alpha_1,\ldots,\alpha_n$ are the Galois conjugates of $\alpha$. The absolute logarithmic height of $\alpha$ is defined by
    \[ 
    \mathbf{h}(\alpha) := \frac{1}{n} \sum_{1\le k\le n} \log\max\{1,|\alpha_k|\}. 
    \]
\item We use $p$, $q$, $\ell$, and their subscripted variants to denote rational primes, and $\mathfrak p$ and its subscripted variants to denote prime ideals.
    \item The sets $\mathcal{P}(E)$, $\mathcal{P}^{\mathrm{ord}}(E)$, and $\mathcal{P}^{\mathrm{ss}}(E)$ are defined in \eqref{eq:allprimes}, \eqref{eq:allordinaryprimes}, and \eqref{eq:allssprimes}, respectively.   
    \begin{center}
\begin{tabular}{l|c}
Counting functions & Defined in equation \\
\hline
$\pi_E(x, \Q(\sqrt{-\Delta}))$ & \eqref{eq:original} \\
$\tilde{\pi}_E(x, \Q(\sqrt{-\Delta}))$ & \eqref{eq:cmtype} \\
$\tilde{\pi}_E(x, \Q(\sqrt{-\Delta}); \delta)$ & \eqref{eq:weaker-bound} \\
$\tilde{\pi}^{\mathrm{ord}}_E(x, \Q(\sqrt{-\Delta}))$ & \eqref{eq:cmtype-ord} \\
$\tilde{\pi}^{\mathrm{ss}}_E(x, \Q(\sqrt{-\Delta}))$  &  \eqref{eq:cmtype-ss} \\
$G_{\delta}(x; D)$ & \eqref{eq:G-delta-D}\\
$G_{\mathcal{G}_\eta}(x)$ & \eqref{eq:upper-anysubset}
\end{tabular}
\end{center}
\end{enumerate}

\subsection*{Acknowledgment}
The author would like to thank Yingkun Li and Yunqing Tang for valuable comments on an earlier draft of this paper and Chantal David for her careful reading of the manuscript and helpful suggestions. The author also thanks Pengcheng Zhang and Yu Fu
for helpful discussions related to this work. The author is
grateful to Qiyuan (Alex) Gu for pointing out that the
$\log\log x$ factor in an earlier version of
\Cref{thm:upper-bound} can be removed.

\subsection*{Disclosure on the Use of AI} The author used ChatGPT (OpenAI) to check grammar, improve the precision and clarity of existing sentences and explanations, and assist in generating code for Tables \ref{tabel} and \ref{tabel-2}. The author independently reviewed all suggestions and incorporated only those deemed appropriate. All mathematical arguments, results, and conclusions were independently verified by the author, who takes full responsibility for the content of the manuscript.




\section{Preliminaries}\label{sec:preliminaries}

In this section, we review the background and give some preliminary results on Hilbert class polynomials, optimal embeddings, Deuring's CM theory,  and the equidistribution of Hecke correspondences. 


\subsection{Endomorphism rings of elliptic curves}\label{sec:end}

For further background on the material in this section, we refer the reader to \cite{MR265369}, \cite[\S 13]{MR890960}, \cite[Chapters 11--14]{MR1028322}, and \cite{MR2695524}.

Let $D>0$ be an integer such that $D\equiv 0, 3\pmod 4$.   We say an elliptic curve over a field $k$ \textit{has CM by $\mathcal{O}_D$} if there is an optimal embedding 
\[
\iota: \mathcal{O}_D \hookrightarrow \End_{\overline{k}}(E),
\]
where optimal  means $(\iota(\mathcal{O}_D)\otimes \Q) \cap \End_{\overline{k}}(E)=\iota(\mathcal{O}_D)$. 
For the general definition of optimal embedding of $\Z$-modules, see, e.g., \cite[Definition 2.1]{MR3431591}.


Let $P_D(X)$ be the Hilbert class polynomial associated to the order $\mathcal{O}_D$. It is the monic irreducible polynomial in $\Z[X]$ whose roots are precisely the $j$-invariants of elliptic curves over $\C$ with CM by $\mathcal{O}_D$. Consequently,  $P_D(X)$ is the minimal polynomial over $K:=\Q(\sqrt{-D})$ of each of its roots, the degree of $P_D(X)$ is $h(D)$, and its splitting field over $K$ is the ring class field of $\mathcal{O}_D$.

\begin{remark}\label{rem:optimal}
 Let $E$ be an elliptic curve over  $\F_p$. If $E$ is ordinary, then $\End_{\overline{\F}_p}(E)=\End_{\F_p}(E)$ and it is an order in an imaginary quadratic field (see, e.g., \cite[Theorem 4.2]{MR265369}). Let $\pi_p(E)$ be the Frobenius endomorphism on $E$.   
We also have the inclusion of orders
\begin{equation}\label{eq:inclusion-orders}
 \Z[\pi_p(E)]\subseteq \End_{\F_p}(E)\subseteq \mathcal{O}_K,
 \end{equation}
 where $K=\Q(\pi_p(E))$. 
Suppose that $E$ has CM by $\mathcal{O}_D$, then $K=\Q(\sqrt{-D})$, and since $\End_{\overline{\F}_p}(E)$ is already an order in $K$,  we have  
\[
\iota(\mathcal{O}_D)=(\iota(\mathcal{O}_D)\otimes\Q) \cap \End_{\overline{\F}_p}(E)=\Q(\sqrt{-D})\cap \End_{\overline{\F}_p}(E)=\End_{\overline{\F}_p}(E).
\]
 Thus the optimal embedding identifies $\mathcal{O}_D$ with  $\End_{\overline{\F}_p}(E)$:
\[
\mathcal{O}_D\simeq \End_{\overline{\F}_p}(E)=\End_{\F_p}(E).
\]

On the other hand, if $E$ is supersingular, then $\End_{\F_p}(E)\subsetneq\End_{\overline{\F}_p}(E)$, and there exist many optimal embeddings of orders into a maximal order of the definite quaternion algebra $B_{p, \infty}$ ramified only at $p$ and $\infty$. Consequently, in the supersingular case one must impose additional conditions on such primes $p$ \cite{MR1040429}.  

\end{remark}

Next, we consider the geometric interpretation of roots of $P_D(X) \pmod p$.  By Deuring's lifting lemma \cite[p.259]{MR5125}, complex multiplication in characteristic $p$ can be lifted to characteristic 0. Hence, roots of $P_D(X)\pmod p$ are $j$-invariants of elliptic curves over $\overline{\F}_p$ with CM by $\mathcal{O}_{D'}$ such that $D=D'f^2$. This is recorded in the following lemma. 

\begin{lemma}\label{lem:deuring1}
 Let $j\in \overline{\F}_p$ be a root of  $P_D(X)\equiv 0\pmod p$. If $E/\overline{\mathbb F}_p$ has $j_E=j$, then $\mathcal{O}_D\hookrightarrow \End_{\overline{\F}_p}(E)$. In particular, $E$ has CM by an order  $\mathcal{O}_{D'}$ with  $D= D'f^2$ for some  $f\in \Z$. 
\end{lemma}

In the following lemma, we show that the conclusion can be strengthened if $E/\F_p$ is  ordinary and $j_E$ is a root of $P_D(X)\pmod p$. 
\begin{lemma}\label{lem:deuring2}
 Let  $E$ be an  ordinary elliptic curve  over $\F_p$. Let $D>0$ be an integer such that $-D\equiv 0, 1\pmod 4$  and  $p\nmid D$.  If  $P_D(j_E)\equiv 0\pmod p$,  then $E$ has CM by $\mathcal{O}_D$. Consequently, for any integer $f\neq \pm 1$ such that $(fD, p)=1$, we have $P_{f^2D}(j_E)\not\equiv 0 \pmod p$. 
\end{lemma}
\begin{proof}
By Deuring's lifting theorem, there exists a CM elliptic curve $\tilde{E}$ defined over a number field $K$ and a prime $\fp\mid (p)$, such that  $\tilde{E}\pmod \fp \simeq E$ and $\End_{\overline{K}}(\tilde{E})\simeq \mathcal{O}_{D}$ with an embedding $ \eta:\End_{\overline{K}}(\tilde{E}) \hookrightarrow \End_{\overline{\F}_p}(E)$.
Since $p\nmid D$ and $E/\F_p$ is an ordinary elliptic curve, by \cite[\S 13,  Theorem 12]{MR890960}, it follows that the embedding $\eta$ is an isomorphism. Hence,  $E$ has CM by $\mathcal{O}_D$. 

Similarly,  if $p\nmid f^2D$ and  $P_{f^2D}(j_E)\equiv 0\pmod p$, then the first part of the proof shows that  $E$ has CM by $\mathcal{O}_{f^2D}$. Since $E$ already has CM by $\mathcal{O}_D$, it follows from  \Cref{rem:optimal} that $\End_{\overline{\F}_p}(E)\simeq \mathcal{O}_{D}\simeq \mathcal{O}_{f^2D}$. Therefore   $f^2D=D$,  and we obtain that $f=\pm 1$, which  is a contradiction. The second claim follows.  
\end{proof}

Finally, we give a  height  bound of Hilbert class polynomial $P_D(X)$ evaluated at a fixed $j$-invariant. 
\begin{lemma}\label{lem:height}
 Assume $D>0$ and  $D\equiv 0,3 \pmod 4$. Let $E/\Q$ be an elliptic curve with $j$-invariant  $j_E$. Then \[
 \log \left|\Num(P_D(j_E))\right|\ll_E  \sqrt{D}(\log D)^2,
 \]
 where $\Num(P_D(j_E))$ denotes the numerator of $P_D(j_E)\in \Q$ written in lowest terms. 
\end{lemma}
\begin{proof}

We apply the logarithmic height bound for Hilbert class polynomials \cite[Theorem 1.2]{MR2476572}:
\[
\mathbf{h}(P_D)\ll \sqrt{D}(\log D)^2, 
\]
where $\mathbf{h}(P_D)$ denotes the logarithmic Weil height of the vector of coefficients of  $P_D(X)$. 
Since  $\deg P_D(X)=h(D)$ and $h(D)\ll \sqrt{D}\log D$, evaluating $P_D(X)$ at $j_E$ and taking the numerator,  we see that 
\[
\log \left|\Num P_D(j_E)\right|\ll_E h(D) + \mathbf{h}(P_D)\ll  \sqrt{D}(\log D)^2.
\]
This completes the proof. 
\end{proof}
\begin{remark}
 See Schoof \cite[Proposition 2.3]{MR1085266} and Fouvry--Murty \cite[Lemma 5]{MR1382477} for a similar upper bound using different methods.
\end{remark} 

\subsection{Equidistribution related to Hecke correspondences}\label{sec:equidistribution}
We keep most of the notation from \cite[Section 3]{MR2017146} and \cite[Section 2]{MR3843371}. Let  $X(1)$ be the coarse moduli scheme over $\operatorname{Spec}\Z$ of generalized elliptic curves. The modular $j$-function induces an isomorphism $X(1)\simeq \mathbb{P}^1_\Z$, and we use this identification throughout.

For a positive integer $N$, we denote by $t_{N}\subset X(1)\times_\Z X(1)$ the Hecke correspondence of order $N$, and let $e_N:=N\prod_{p\mid N}(1+1/p)$.  Hecke correspondences also induce a map  $t_{N*}$ from closed points of  $X(1)$ to the divisor group of $X(1)$. For a complex $j$-invariant $y$, we write the $N$-th Hecke orbit of $y$ as
\[
t_{N*}y:= j_1+\cdots+j_{e_N}, 
\]
where the sum runs over the $j$-invariants of the  $N$-Hecke points of $y$. The set of $N$-Hecke points associated to $y$ is denoted by $|t_{N*}y|$ and for each $\alpha\in |t_{N*}y|$,  we denote by  $j_\alpha$ its  $j$-invariant. Using the identification $X(1)(\C)\simeq \SL_2(\Z)\backslash\overline{\mathbb{H}}$, where   $\mathbb{H}$ is the upper half plane and $\overline{\mathbb{H}}=\mathbb{H}\cup (\Q\cup \{\infty\})$,  we denote by $\tau_\alpha$ the image of $\alpha$ in $\SL_2(\Z)\backslash\overline{\mathbb{H}}$. 

Now fix a metrized line bundle on  $X(1)$. Denote  $\widehat{\mathcal{L}}:=(\mathcal{L}, \|\cdot \|)$ the metrized line bundle with $\mathcal{L}$ the usual line bundle of modular forms of weight 12 on $X(1)$. The modular discriminant $\Delta(\tau)$ induces a global section on $\mathcal{L}$. Let $Z$ be a horizontal divisor of relative degree $d'$ on $X(1)_\Z$ with complex fiber $Z_\C=\sum_{1\le i\le k} n_i Q_i$ with $Q_i\in X(1)(\C)$.  Assuming $j(Q_i)\neq \infty$, we get a rational section $s_Z$ on  $\mathcal{L}^{\otimes d'}$:
\[
s_Z(\tau):=\prod_{1\le i\le k} \left((j(Q_i)-j(\tau))\Delta(\tau)\right)^{n_i}.
\]
Finally,  for a purely horizontal divisor $Y$ of degree $d$ on $X(1)_{\Z}$ with $Y_\C:=\sum_{1\le j\le t} m_jP_j$  and $P_j\in X(1)(\C)$, define 
  \[
  s_Z(Y):=\prod_{1\le j\le t} s_Z(P_j)^{m_j}.
  \]
  The $L_1^2$ log-singular Hermitian metric $\| \cdot \|$ on $\mathcal{L}_\C$ is induced by the  Petersson norm on modular forms that satisfies  
\[
\| \Delta(\tau) \|:= |\Delta(\tau)||\Im(\tau)|^6,
\]
where $|\cdot |$ is the complex absolute value. 
For the horizontal divisors $Z:=\sum_{1\le i\le k}n_i Q_i$ and $Y=\sum_{1\le j\le t} m_jP_j$, denote  
\[
\| s_Z(Y)\|=\prod_{1\le j\le  t}\prod_{1\leq i\leq k }\left(|j(Q_i)-j(P_j)| \|\Delta(P_j)\|\right)^{n_im_j}.
\]

For a zero-cycle $R:=\sum_{s} w_sR_s$ on $X(1)_\Z$, where $R_s$ are closed points of $X(1)_\Z$, we define the (finite) arithmetic degree  of $R$ as 
\[
\widehat{\deg}(R):=\sum_{s}w_s\log \# \kappa(R_s),
\]
where $\kappa(R_s)$ denotes the residue field of $R_s$. 
For two divisors $Y, Z$ on $X(1)$,  we denote $\widehat{\deg}(Z.Y)$ the arithmetic degree of the   0-cycle $Z.Y$. This is a sum of local contributions $\deg_p(Z.t_{N*}Y)$ over all rational primes $p$, which is 0 for almost all $p$. If $Y$ and $Z$ are Zariski closures of   $y, z\in X(1)_\C$ in $X(1)_\Z$, then for a finite place $p$, there is  \cite[p.2047]{MR3843371} 
\begin{equation}\label{eq:local-degree}
\deg_p(Z.t_{N*}Y)=-\sum_{\alpha\in |t_{N*}y|}\log \left(\frac{|j_\alpha-z|_p}{\max\{1, |j_\alpha|_p\}\cdot \max\{1, |z|_p\}} \right) ,
\end{equation}
where $|\cdot |_p$  is the corresponding normalized, non-archimedean
absolute value defined on $\widehat{{\Q}^{\text{ur}}_p}$ so that $|p|_p=1/p$.  

We recall a result of Autissier \cite[Th\'eor\`eme 3.2]{MR2017146} and Charles  \cite[Corollary 2.2]{MR3843371}, which plays an important role in the proof of \Cref{thm:upper-bound}.

\begin{theorem}\label{prop:Charles}
 Let $Y, Z$ be purely horizontal divisors on $X(1)_{\Z}$ of relative degrees $d$ and $d'$, respectively. Assume that $Y$ is effective and irreducible, and for every positive integer $N$,  the divisors $t_{N_*}Y$ and $Z$  have no common component. Then, for $N\geq 3$,  
\begin{equation}\label{eq:Charles}
    \widehat{\deg}(Z.t_{N*}Y)- \log||s_Z(t_{N*}Y)||= 6dd' e_N\log N+O\left(dd'e_N\left(\log\log N+h_{\widehat{\mathcal{L}^{\otimes d'}}}'(Y)\right) \right),
    \end{equation}
    where $h'_{\widehat{\mathcal{L}^{\otimes d'}}}(Y)$ is the normalized height of $Y$ relative to $\widehat{\mathcal{L}^{\otimes d'}}$.
\end{theorem}
The left-hand side of \eqref{eq:Charles} is a global height pairing and sometimes called the \textit{global intersection number}. It decomposes into local contributions: the non-archimedean contributions are given by the local intersection multiplicities on the fibers at finite primes
 $p$, while the archimedean contribution at $\infty$ is given by $-\log||s_Z(t_{N*}Y)||$. 

\section{Singular moduli}\label{sec:singular}

In this section, we introduce singular moduli and \Cref{conj:super}. We then prove several key lemmas concerning both the archimedean and non-archimedean properties of differences between singular moduli and a non-CM integral $j$-invariant. 
These results will be used to prove that the set of supersingular primes in $\mathcal{P}(E)$ is  infinite under one of the alternatives in \Cref{prop:nonempty}. We refer the reader to \cite{MR4545213, MR4713026} for further background on this section.

\subsection{Non-archimedean properties}\label{sec:non-arch}
Let $E/\Q$ be a non-CM elliptic curve. Let $D>0$ be an integer such that $-D\equiv 0, 1\pmod 4$. 
 A \emph{singular modulus of
    discriminant $-D$} is the $j$-invariant of an elliptic curve over
    $\C$ with complex multiplication by $\mathcal{O}_D$.
Let $v$ be a place of $\Q$. If $v=p$ is a finite prime, denote by $\mathbb{C}_v$ a completion of an algebraic closure of the field of $p$-adic numbers $\Q_p$ with the normalized $p$-adic absolute value  $|\cdot|_p$; if $v=\infty$, then take $\C_\infty=\C$ and  $|\cdot|_\infty$ is the  usual absolute value on $\C$. Fix an embedding $\overline{\Q}\hookrightarrow \C_v$  throughout.

For each $v$, we say a number $\alpha\in \C_v$ is  \textit{badly approximable} in $\C_v$ by singular moduli, if there exist constants $A(\alpha, v)>0$ and $B(\alpha, v)$ such that for every singular modulus $j\neq \alpha$ with discriminant $-D$, we have 
\[
-\log|j-\alpha|_v\leq A(\alpha, v)\log h(D)+B(\alpha, v).
\]
If $\alpha$ is a singular modulus, the result  \cite[Theorem C]{MR4713026} by Herrero et al. implies that for each $v$,  $\alpha$ is badly approximable in $\C_v$ by singular moduli.    They also noted that when $v=\infty$,  Habegger \cite{MR3404647} proved that every algebraic number $\alpha$ is badly approximable in $\C_\infty$ by the singular moduli. This motivates them to make a conjecture on an analogous result in the $p$-adic setting \cite[Conjecture 3]{MR4713026}. 

We now formulate our $p$-adic approximation conjecture in two forms, where the first version is a special case of  \cite[Conjecture 3]{MR4713026}. 

\begin{conjecture}\label{conj:super}
   Let $p$ be a prime number. 
   \begin{enumerate}
       \item  Every algebraic number is badly approximable in $\C_p$ by singular moduli.
  \item (Uniform Version) Every algebraic number  $\alpha$ is badly approximable by singular moduli in $\C_p$ on average.
  More precisely, for every fixed algebraic number \(\alpha\), there exist constants $C(\alpha)>0$ and $\mu\in (0, 1)$ such that
  \[
-\sum_{1\le i\le h(D)}\log|j_i-\alpha|_p\leq h(D) C(\alpha) (\log p)^\mu,
\] 
where $j_i\neq \alpha$ are the CM $j$-invariants with CM by $\mathcal{O}_D$.
   \end{enumerate}
\end{conjecture}

Roughly speaking, the uniform version of \Cref{conj:super} provides quantitative control on how closely singular moduli can approximate a fixed algebraic number $p$-adically on average. In particular, the average $p$-adic distance of the CM $j$-invariants to $\alpha$ is controlled uniformly in $D$.
 Related $p$-adic approximation and distribution results for singular moduli have been established in several settings
\cite{MR3404647,MR4190395,HeMeLe2021,MR4251608,MR4545213,MR4296371}. In particular, if $\alpha$ is itself a singular modulus, then it is known that $A(\alpha, p)=O_\alpha((\log p)^{-1})$ and $B(\alpha, p)=O_\alpha(\log p)$  (see \cite[Section 4.3]{MR4713026} and \cite[Theorem 3.1]{MR4545213}).

Conjectures of this type are closely related in spirit to the $p$-adic distribution of CM points, which provides a non-archimedean counterpart to the classical archimedean theory  (see \cite{MR4129386,HeMeLe2021,MR4527824} and the references therein). They may also be viewed from the perspective of $p$-adic unlikely intersections. \Cref{conj:super} predicts a quantitative restriction on how closely CM points of $X(1)$ can approach a fixed non-special algebraic point in the $p$-adic topology. This is reminiscent of the  phenomena underlying the Andr\'e--Oort conjecture.

The next result may be viewed as a special case of \Cref{conj:super}, in which we restrict to a particular family of odd supersingular primes and a special family of discriminants of the singular moduli.
The proof relies on results from \cite{HeMeLe2021}, and we follow some of the notation introduced there. For convenience, we recall the relevant notation and basic facts, following \cite[Section 2]{HeMeLe2021}.

Recall that $Y(\C_p)$ denotes the moduli space of elliptic curves over $\C_p$, and that $Y_{\mathrm{ss}}(\C_p)$ denotes the subset of elliptic curve classes in $Y(\C_p)$ having supersingular reduction at $p$. A $p$-adic quadratic order is a $\Z_p$-order in a quadratic field extension of $\Q_p$. A $p$-adic discriminant $\mathfrak{D}$ is the set of discriminants of a $p$-adic quadratic order obtained by varying over all $\Z_p$-bases, and may thus be viewed as a subset of $\Q_p^\times/(\Z_p^\times)^2$.  Every $p$-adic discriminant is of the form
$p^{2m}\mathfrak{d},$ for some $m\geq 0$ and some fundamental $p$-adic discriminant $\mathfrak{d}$. See \cite[Lemma A.1(iii)]{HeMeLe2021} for the classification of fundamental $p$-adic discriminants.

For each $E\in Y(\C_p)$ having smooth reduction at $p$, let $\mathcal{F}_E$ denote its formal group over $\overline{\Q}_p$. If the endomorphism ring $\End(\mathcal{F}_E)$ is a $p$-adic quadratic order, we say that $E$ has \emph{formal complex multiplication} and call $E$ a \emph{formal CM point}.

For each $p$-adic discriminant $\mathfrak{D}$, let $\Lambda_{\mathfrak{D}}$ denote the set of elliptic curves $E\in Y(\C_p)$ that are formal CM points of $p$-adic discriminant $\mathfrak{D}$.  
By \cite[Theorem 4.11]{HeMeLe2021}, every formal CM point has supersingular reduction. Consequently,
\[
\Lambda_{\mathfrak{D}}\subseteq Y_{\mathrm{ss}}(\C_p).
\]
Moreover, for any two distinct $p$-adic discriminants $\mathfrak{D}_1$ and $\mathfrak{D}_2$, the compact sets $\Lambda_{\mathfrak{D}_1}$ and $\Lambda_{\mathfrak{D}_2}$ are disjoint.

\begin{lemma}\label{lem:ss-conj}
Let $E/\Q$ be a non-CM elliptic curve and let $-D<0$ be a fundamental discriminant. 
Let $p$ be an odd supersingular prime of $E$. Then there exists $m_E(p, D)\geq 0$ such that for all even integers $m\geq m_E(p, D)$ and all sufficiently large $f$ with $(f,p)=1$,  
\[
\max_{1\le i\le h(p^{m}f^2D)} -\log |j_E-j_i|_p\leq C(m, E, p, D),
\]
where $j_i$ are $j$-invariants of CM elliptic curves with CM by $\mathcal{O}_{p^mf^2D}$ and $C(m, E,p,D)$ is a constant depending only on $m, E,p$, and $D$ but not on $f$.
As a consequence, 
\begin{equation}\label{eq:ss-conj}
\frac{1}{ h(p^{m}f^2D)}\sum_{1\le i\le h(p^{m}f^2D)}
-\log |j_E-j_i|_p
\leq
C(m, E,p,D).
\end{equation}

\end{lemma}

    \begin{proof}
   We base change $E$ to $\Q_p$, and continue to denote the resulting elliptic curve by $E$. Since $p$ is a supersingular prime of $E$, we have 
   $
   E\in Y_{\mathrm{ss}}(\Q_p)\subseteq Y_{\mathrm{ss}}(\C_p).
   $ 
   
   We first choose $m_E$. If $E$ is not a formal CM point,  then $E\notin\Lambda_{\mathfrak D}$
for every $p$-adic discriminant $\mathfrak D$, and we set $m_E=2$. Suppose instead that $E$ is a formal CM point. Since the sets $\Lambda_{\mathfrak D}$ associated to distinct $p$-adic discriminants $\mathfrak D$ are disjoint, there is a unique $p$-adic discriminant  
\[
\mathfrak D_E=p^{2m_0}\mathfrak d_E,
\qquad m_0\geq0,
\]
such that $E\in\Lambda_{\mathfrak D_E},$ where $\mathfrak d_E$ is a  fundamental $p$-adic discriminant. In this case, set $m_E=2m_0+2.$
Therefore, in either case, for every fundamental $p$-adic discriminant $\mathfrak d$, we have
$
E\notin\Lambda_{p^{m}\mathfrak d}
$
for all even integers $m\geq m_E$.

For $f\geq 1$ with $(f,p)=1$, set
$
D_f:=p^{m}f^2D.
$
 If $p\nmid P_{D_{f}}(j_E)$, then $|j_E-j_i|_p=1$ for every $1\leq i\leq h(p^mf^2D)$, hence the left hand side of \eqref{eq:ss-conj} is 0 and the result follows. 
 
    Assume now $p\mid P_{D_{f}}(j_E)$. 
    Then, there is a CM point $j_i$ of discriminant \(-D_f\) whose reduction agrees with that of \(j_E\). Since \(E\) has supersingular reduction at \(p\), this CM point has supersingular reduction. In particular, $-D_f$ determines a $p$-adic discriminant $\mathfrak D_K$ \cite[Corollary 4.12]{HeMeLe2021}.
Since $(f,p)=1$, $p$ is odd, and $m$ is even, the corresponding $p$-adic discriminant is 
$
p^{m}\mathfrak d_K
$, where $\mathfrak d_K$ is a fundamental $p$-adic discriminant.  Consequently, all CM points $j_i$ with $1\leq i\le h(p^mf^2D)$ of discriminant $-D_f$, after base change to $\overline{\Q}_p$, belong to
$\Lambda_{p^{m}\mathfrak d_K}.$
Since $\Lambda_{p^{m}\mathfrak d_K}$ is a compact subset of $Y(\C_p)$ \cite[Theorem A]{HeMeLe2021} and $j_E\notin \Lambda_{p^{m}\mathfrak d_K}$, we get the first assertion. The second assertion follows trivially. 

    \end{proof}

\subsection{Archimedean properties}
The main result of this section is \Cref{lem:notpm1}, which is the key lemma  used  in the proof of \Cref{prop:nonempty}. 
This result may be of independent interest, as it is closely related to the problem of showing that singular moduli, and more generally differences of singular moduli, are not units \cite{MR4190395,MR4251608}. The proof uses the ``hard" lower bound for the Weil height of singular moduli  \cite[Proposition 4.3]{MR4190395} with the following lemma, which shows that the average value of $\log|j-j_E|$ as $j$ ranges over singular moduli of a fixed discriminant admits the following lower bound. 

\begin{lemma}\label{lem:singular-moduli}
  Let $E/\Q$ be a non-CM elliptic curve and $-D$ be a fundamental discriminant. For each $f\in \Z_{\geq 1}$ denote  $D_f:=f^2D$. Let $1\le t\le h(D_f)$, and $\{j_{m_1}, \ldots, j_{m_t}\}$ be a subset of singular moduli of discriminant $-D_f$, such that 
  \[
  |j_{m_k}-j_E|<1, \qquad 1\leq k\leq t.
  \] 
  Then, for any $0<\eta<1$, there exists a constant $c_0$ depending only on $E, D$, and $\eta$, such that for all $f>c_0$,
  \[
  \sum_{1\leq k\leq t} \log |j_{m_k}-j_E|\geq -\eta  h(D_f)\log D_f.
  \]
\end{lemma}
\begin{proof}
 For each $1\le k\le t$, let $\tau_E$ and $\tau_{m_k}\in \mathbb{H}$ be such that $j(\tau_E)=j_E$ and $j(\tau_{m_k})=j_{m_k}$, respectively.  

Since $j_E\neq 0, 1728$, we have  $j'(\tau_E)\neq 0$. By the inverse function theorem,  there exist a neighborhood $U_E$ of $\tau_E$ and a constant $c_{\mathcal{J}}>0$ (depending only on $E$) such that  $j(\tau)$ is analytic on $U_E$ and 
\begin{equation}\label{eq:j-tau-relation}
    |\tau_{m_k}-\tau_E|\le c_{\mathcal{J}}  |j_{m_k}-j_E|.
\end{equation} 
Assume $j(U_E)\supseteq B_{j_E}(r_E)$ for some $r_E$ satisfying $0<r_E<1$, where $B_{j_E}(r_E)\subseteq \mathbb{C}$ is the ball of radius $r_E$ around $j_E$.  


Let $\{n_1, \ldots, n_{t'}\}$ be a subset of $\{m_1,\ldots, m_t\}$ such that  $|j_{n_k}-j_E|\geq r_E$ for every $1\le k\le t'$.  Then 
\[
\sum_{1\le k\le t'} \log |j_{n_k}-j_E|\geq \log r_E\cdot h(D_f).
\]
This is already stronger than the desired bound we want to prove. Therefore, it suffices to assume that
$|j_{m_k}-j_E|<r_E$ for every $1\leq k\leq t$ and derive
the desired lower bound under this assumption.

First, for each $1\leq k\leq t$,  we give a ``trivial" lower bound for $|j_{m_k}-j_E|$. By applying \cite[Lemma 5, Lemma 8, and Eq (11)]{MR3404647} (or \cite[Theorem C]{MR4713026}), it follows that there exists a constant $c_1>0$, depending only on $E$, such that 
\begin{align}\label{eq:least-lower-bound}
\log |j_E-j_{m_k}|\ge -c_1\log D_f.
\end{align}

Next, we show that for most $k\in\{1,\ldots,t\}$, there is a better lower bound. For any $\epsilon>0$, denote 
\[
\mathcal{J}(\epsilon):=\{1\le k\le t: |j_{m_k}-j_E|\leq \epsilon\}, \;\;  \mathcal{D}(\epsilon):=\{1\le k\le t: |\tau_{m_k}-\tau_E|\leq \epsilon\}.
\]
Using \eqref{eq:j-tau-relation}, it follows that 
\begin{equation}\label{eq:tau-j}
\#\mathcal{J}(\epsilon/c_{\mathcal{J}})\leq \#\mathcal{D}(\epsilon).
\end{equation}
By \cite[Corollary 3.2]{MR4276350}, we obtain that for all sufficiently large $f$ and $0<\epsilon<1/4$, 
\[
\#\mathcal{D}(\epsilon)\leq 50F\left(D_f\right)\left(\epsilon^2  D_f^{\frac{1}{2}}\log\log (D_f^{\frac{1}{2}})+ \epsilon D_f^{\frac{1}{2}}+1 \right)
\]
where 
\[
F(D_f):=\max\{2^{\omega(n)}: n\leq D_f^{\frac{1}{2}}\},
\]
and $\omega(n)$ is the number of distinct prime divisors of $n$. 

 Using the bound $2^{\omega(m)}\leq \exp(\frac{1.3841\log m}{\log\log m})$ \cite[Th\'eor\`eme 11]{MR736719}
and taking $\epsilon=D_f^{-\eta}$,  we obtain that for sufficiently large $f$ (depending only on $D$ and $\eta$) 
\begin{align}\label{eq:number-bad-bound-is-small}
\#\mathcal{D}(\epsilon) & \leq 50\exp\left(\frac{\log D_f}{\log\log D_f^{\frac{1}{2}}}\right)\left(D_f^{\frac{1}{2}}\exp\left(\log \log\log (D_f^{\frac{1}{2}})-2\eta(\log D_f)\right) \right. \nonumber \\
& \hspace{200pt} \left.+D_f^{\frac{1}{2}}\exp(-\eta(\log D_f))+1\right) \nonumber \\
& \leq 50\left(D_f^{\frac{1}{2}}\exp\left(\log \log\log (D_f^{\frac{1}{2}})+\frac{\log D_f}{\log\log D_f^{\frac{1}{2}}}-2\eta(\log D_f)\right) \right. \nonumber \\
& \hspace{100pt} \left.+D_f^{\frac{1}{2}}\exp\left(\frac{\log D_f}{\log\log D_f^{\frac{1}{2}}}-\eta(\log D_f)\right)+\exp\left(\frac{\log D_f}{\log\log D_f^{\frac{1}{2}}}\right)\right) \nonumber \\
& \leq 50\left(D_f^{\frac{1}{2}}\exp\left(-\eta(\log D_f)\right)+D_f^{\frac{1}{2}}\exp\left(-\frac{\eta(\log D_f)}{2}\right)+\exp\left(\frac{\log D_f}{\log\log D_f^{\frac{1}{2}}}\right)\right) \nonumber \\
& \leq 150 D_f^{\frac{1}{2}}\exp\left(-\frac{\eta(\log D_f)}{2}\right)=150 D_f^{\frac{1-\eta}{2}}.
\end{align}

Therefore for $\epsilon=D_f^{-\eta}$, using \eqref{eq:least-lower-bound}, \eqref{eq:tau-j},  and \eqref{eq:number-bad-bound-is-small}, we obtain 
\begin{align}\label{eq:GRH}
    \sum_{1\leq k\leq t}\log |j_{m_k}-j_E| & =   \sum_{\substack{1\le k\le t \\ k\notin \mathcal{J}(\epsilon/c_{\mathcal{J}})}}\log |j_{m_k}-j_E|+ \sum_{\substack{1\le k\le t \\ k\in \mathcal{J}(\epsilon/c_{\mathcal{J}})}}\log |j_{m_k}-j_E| \nonumber \\
    & \geq   h(D_f)\log (D_f^{-\eta}/c_{\mathcal{J}}) -c_1 \#\mathcal{J}(\epsilon/c_{\mathcal{J}}) \log D_f \nonumber \\
   &  \geq -\eta  h(D_f)\log D_f + h(D_f)\log (1/c_{\mathcal{J}})-150c_1 D_f^{\frac{1-\eta}{2}}\log D_f.
\end{align}
By the class number formula \cite[Corollary 7.28]{MR1028322} and Mertens' theorem,
\begin{align}\label{eq:class-number}
h(D_f)&=\frac{h(D)f}{[\mathcal{O}^\times_D:\mathcal{O}^\times_{f^2D}]}\prod_{\ell\mid f}\left(1-\kronecker{-D}{\ell}\frac{1}{\ell}\right) \nonumber \\
&\gg_D f \prod_{\ell\mid f}\left(1-\frac{1}{\ell}\right) \nonumber\\
& \geq c(D) \frac{D_f^{1/2}}{\log\log D_f}
\end{align}
for some constant $c(D)$.
 Hence, \eqref{eq:GRH} becomes
\begin{align*}
\geq & -\eta  h(D_f)\log D_f + h(D_f)\log (1/c_{\mathcal{J}})-\frac{150c_1}{c(D)} h(D_f)(\log D_f \log\log D_f) D_f^{-\frac{\eta}{2}}\\
\geq &  -2\eta  h(D_f)\log D_f
\end{align*}
 for all $f>c_0$, where $c_0$ is an ineffective constant that depends only on $E, D$ and $\eta$. The desired bound follows by replacing the above $\eta$ by $\eta/2$. 
\end{proof}

Now, we are ready to derive the key lemma.

\begin{lemma}\label{lem:notpm1}
   Let $E/\Q$ be a non-CM elliptic curve. Let $-D$ be a fundamental discriminant. Then, there exists a constant $A(E, D)>0$ such that for all sufficiently large $f$,
   \[
   |\Num(P_{D_f}(j_E))| := \Num\left(\prod_{1\leq k\leq h(f^2 D)}\left|j_E-j_k\right|\right)\gg  \exp(A(E, D)h(D_f)\log D_f), 
   \]
   where  $D_f:=f^2D$.
   In particular, there exists a constant $c(E, D)>0$ depending only on $E$ and $D$, such that for all  $f>c(E, D)$,  
$P_{D_f}(j_E)\ne \pm 1,$
where $j_k$ are $j$-invariants of CM elliptic curves with CM by $\mathcal{O}_{f^2D}$.
\end{lemma}

    \begin{proof}
We give the proof for $j_E\in \Z$, since the general case follows by clearing the denominator. Recall that 
  \begin{equation}\label{eq:product}
P_{D_f}(j_E) := \prod_{1\leq k\leq h(D_f)}(j_E-j_k).
\end{equation}
By the ``strong" lower bound for the absolute logarithmic height of a singular modulus \cite[Proposition 4.3]{MR4190395}, we have  that for any $1\leq k\leq h(D_f)$,
\begin{equation}\label{eq:high-lower-bound}
\mathbf{h}(j_k):=\frac{1}{h(D_f)}\sum_{1\le s\le h(D_f)}\log\max\{1, |j_s|\}\geq \frac{3}{\sqrt{5}}\log D_f-9.79,
\end{equation}
where $|\cdot|$ is the usual complex absolute value.

Using
\[
\max\{1,|j_s    |\}
\le (1+|j_E|)\max\{1,|j_s-j_E|\},
\] 
we obtain
\begin{align*}
&\frac{1}{h(D_f)}\sum_{1\le s \le h(D_f)}\log\max\{1, |j_s|\}\\
 \leq & \frac{1}{h(D_f)} \sum_{1\le s\le h(D_f)}\log \left(\max\{1, |j_s-j_E|\} (|j_E|+1)\right).
\end{align*}
After reordering the singular moduli, let $k_0':=k_0'(f)\geq 1$ be the integer such that 
\[
|j_E-j_{k}|\geq 1 \text{ for $ k< k_0'$}  \quad \text{ and } \quad   |j_E-j_{k}|<1 \text{ for $k\geq k_0'$.}
\]

Together with \eqref{eq:high-lower-bound}, we obtain that for all sufficiently large $f$,
\begin{align*}
 &\frac{3}{\sqrt{5}}\log |D_f|-9.79 
  \leq  \frac{\sum_{1\le s\le k_0'-1}\log |j_s-j_E|}{h(D_f)}+\log \left(1+|j_E|\right).
\end{align*}
For $f$ sufficiently large, we see that the last two terms in the upper bound above are $O_E(1)$. In particular,  we derive that $k'_0\geq 2$ and 
\begin{equation}\label{eq:sum-lower-bound}
    \sum_{1\le s\le k_0'-1}\log |j_s-j_E| \geq \left(\frac{3}{\sqrt{5}}-o_E(1)\right){h(D_f)}\log D_f\geq \frac{2}{\sqrt{5}}{h(D_f)}\log D_f
\end{equation}
for all sufficiently large $f$  (depending on $E$ and $D$).

On the other hand,  applying \Cref{lem:singular-moduli} to the subset  $\{j_{k_0'}, \ldots, j_{h(D_f)}\}$ with  $\eta=\frac{1}{\sqrt{5}}$
gives
\begin{align}\label{eq:final-lower-bound}
  \sum_{k_0'\le s\le h(D_f)}\log |j_s-j_E|  \geq -\frac{1}{\sqrt{5}} h(D_f) \log D_f.
\end{align}

Putting the two bounds \eqref{eq:sum-lower-bound} and \eqref{eq:final-lower-bound} together, we see that 
\begin{align*}
|P_{D_f}(j_E)| & = \prod_{1\leq k\leq h(D_f)}|j_E-j_k|\\
& \geq \exp\left(\frac{2}{\sqrt{5}}h(D_f)\log |D_f|-\frac{1}{\sqrt{5}}h(D_f) \log D_f\right)\\
& \geq \exp\left(\frac{1}{\sqrt{5}} h(D_f)\log D_f\right)
\end{align*}
 holds if $f$ is sufficiently large (depending on $E$ and $D$). 
The first lower bound now follows. 
 
 Hence, $|P_{D_f}(j_E)|\to \infty $ as $f\to \infty$,   and we derive that $|P_{D_f}(j_E)|\neq 1$ when $f$ is sufficiently large (depending only on $E$ and $D$). The result follows since $P_{D_f}(j_E)$ is an integer. 
\end{proof}

\begin{remark}
\begin{enumerate}

    \item \Cref{lem:notpm1} can be viewed as a quantitative version of a special case of \cite[Theorem 3.2]{aslanyan2023multiplicative}. Their result implies that for any fixed $j_E\in \C\backslash\{0, 1728\}$ and  $n\geq 1$, there \textit{exist} only finitely many $n$-tuples of singular moduli $(j_1, \ldots, j_n)$ satisfying 
    \[
    \prod_{1\leq k\leq n}(j_E-j_k)=\pm 1, 
    \]
  which is  minimal in the sense that no proper subproduct is equal to $\pm 1$.
    However, the proof relies on an ineffective result in  \cite[Theorem 1.6]{MR4371528}.  Hence, it does not provide an explicit bound for the number of such $n$-tuples, nor does it determine them explicitly.
    \item 
   Combining \Cref{lem:height} and \Cref{lem:notpm1} and using \eqref{eq:class-number}, we see that for a fixed  $D$ as above and any sufficiently large $f$, 
    \[
\frac{f\log f}{\log\log f} \ll_{E, D} \log|\Num(P_{f^2D}(j_E))|\ll_{E, D} f(\log f)^2 \quad \text{ as $f\to \infty$.}
    \]

\end{enumerate}

\end{remark}

\section{The lower bound}\label{sec:Theorem2}

In this section, we prove the set $\mathcal{P}(E)$ defined in \eqref{eq:allprimes} 
is infinite,  obtain a lower bound for $\tilde{\pi}_E(x, \Q(\sqrt{-\Delta}))$, and prove \Cref{thm:lower-bound}.

\subsection{Set Up}\label{sec:setup}

In \Cref{sec:setup} we introduce some notation and make a few observations related to the  quadratic twists of elliptic curves.

Let $E/\Q$ be a non-CM elliptic curve and $\Delta\geq 1$ be a fixed squarefree integer satisfying the following assumptions in \Cref{thm:lower-bound}. Set  $K=\Q(\sqrt{-\Delta})$ and define
\begin{equation}\label{eq:deltastar}
    \Delta^*:=\begin{cases}
    \Delta & \text{ if $\Delta\equiv 3\pmod 4$}\\
    4\Delta & \text{ if $\Delta\equiv 1, 2\pmod 4$}
\end{cases}.
\end{equation}
Note that 
$
\disc(\mathcal{O}_K)=-\Delta^*,
$ 
and that assumption \eqref{assump:Delta} in \Cref{thm:lower-bound}  implies $\Delta^*=\Delta$.

We split the set $\mathcal{P}(E)$ into 
\begin{equation}\label{eq:allordinaryprimes}
\mathcal{P}^{\mathrm{ord}}(E):=\{p\in \mathcal{P}(E): \text{ $p$ is an ordinary prime} \}
\end{equation}
and 
\begin{equation}\label{eq:allssprimes}
\mathcal{P}^{\mathrm{ss}}(E):=\{p\in \mathcal{P}(E): \text{ $p$ is a supersingular prime}\}.
\end{equation}
Denote 
\begin{equation}\label{eq:cmtype-ord}
\tilde{\pi}^{\mathrm{ord}}_E(x, \Q(\sqrt{-\Delta})):=\#\{p\leq x: p\in \mathcal{P}^{\mathrm{ord}}(E)\}
\end{equation}
and 
\begin{equation}\label{eq:cmtype-ss} \tilde{\pi}^{\text{ss}}_E(x, \Q(\sqrt{-\Delta})):=\#\{p\leq x: p\in \mathcal{P}^{\mathrm{ss}}(E)\}.
\end{equation}
 Observe that
 \[
 \tilde{\pi}^{\mathrm{ord}}_E(x, \Q(\sqrt{-\Delta}))=\#\{p\leq x: p\nmid N_E, \text{ $p$ is an ordinary prime, } \Q(\pi_p(E_p))=\Q(\sqrt{-\Delta})\}.
 \]
 This is because if $E_p$ is ordinary, then $\End_{\overline{\F}_p}(E_p)\otimes_\Z \Q\simeq \Q(\pi_p(E_p))$, which can be seen from \eqref{eq:inclusion-orders}. Since there are only finitely many supersingular primes counted by $\pi_E(x, \Q(\sqrt{-\Delta}))$, we get 
 \begin{equation}\label{eq:compare-with-tilde}
  \tilde{\pi}^{\mathrm{ord}}_E(x, \Q(\sqrt{-\Delta}))\ll  \pi_E(x, \Q(\sqrt{-\Delta})).
 \end{equation}

This separation will be helpful in the proof of \Cref{thm:lower-bound}.
Hence, to show $\mathcal{P}(E)$ is infinite, it suffices to show either $\mathcal{P}^{\mathrm{ord}}(E)$  or $\mathcal{P}^{\mathrm{ss}}(E)$ is infinite.

Next, we note that to show $\mathcal{P}(E)$ is infinite, it suffices to show $\mathcal{P}(E')$ for a quadratic twist  $E'/\Q$ of $E$ is infinite.  This is because for any  prime $p\nmid N_EN_{E'}$, we have $a_p(E)=\pm a_p(E')$,  and hence 
 \[
 \Q(\pi_p(E_p))=\Q(\sqrt{a_p(E)^2-4p})=\Q(\sqrt{a_p(E')^2-4p})=\Q(\pi_p(E'_p)).
 \]
 In particular, if $p$ is an ordinary prime of $E$, then it is also an ordinary prime of \(E'\).  So if $p\in \mathcal{P}^{\mathrm{ord}}(E)$ and $p\nmid N_{E'}$, then $p\in \mathcal{P}^{\mathrm{ord}}(E')$. If $p\nmid N_EN_{E'}$ and $p>3$, and it is a supersingular prime of $E$, then  $a_p(E)=a_p(E')=0$, and hence it is also a supersingular prime of $E'$. If $p=2$ (resp. $p=3$), then it is a supersingular prime of $E$ if and only if $a_p(E)\in \{0, \pm 2\}$ (resp. $a_p(E)\in \{0, \pm 3\}$). Hence $p$ is also a supersingular prime of $E'$.   In this case, \[
 \End_{\overline{\F}_p}(E_p)\otimes_\Z \Q\simeq \End_{\overline{\F}_p}(E'_p)\otimes_\Z \Q
 \]
 and they are both isomorphic to the unique definite quaternion algebra over  $\Q$ ramified only at $p$ and $\infty$. Therefore, if $p\in \mathcal{P}^{\mathrm{ss}}(E)$ and $p\nmid N_{E'}$, then $p\in \mathcal{P}^{\mathrm{ss}}(E')$. 
  From the discussion above, and by symmetry of quadratic twists, we have 
  \[
  \mathcal{P}(E')\subseteq \mathcal{P}(E)\cup \{p: p\mid N_E\}
  \]
  for any quadratic twist $E'$ of $E$.  
Since quadratic twisting can therefore change these sets only at the finitely many primes dividing $N_E$, we write $\mathcal{P}$, $\mathcal{P}^{\mathrm{ord}}$, and $\mathcal{P}^{\mathrm{ss}}$ for the unions of the corresponding sets over all quadratic twists of $E$.

 Finally, we explain the role of the assumption in \Cref{thm:lower-bound} that, for every odd prime \(p\mid N_E\), the Kodaira symbol of \(E\) at \(p\) is \(\operatorname{I}_0^*\).   Note that the condition \(j_E\in \mathbb{Z}\) is equivalent to \(E\) having potentially good reduction at every prime; in particular, this holds at every prime \(p\mid N_E\).  Viewing $E$ as an elliptic curve over the local field $\Q_p$,  Tate's algorithm \cite{MR393039} shows that  if $p\geq  3$,  then the Kodaira symbol at $p$ is one of $\operatorname{I}_0^*, \operatorname{II}, \operatorname{II}^*, \operatorname{III}, \operatorname{III}^*, \operatorname{IV}, \operatorname{IV}^*$. 
Moreover, by the proof of  \cite[Proposition 2]{MR1295951},   
if $p\geq 3$ and the Kodaira symbol at $p$ is $\operatorname{I}_0^*$,  then there exists a local quadratic character $\chi_p$ ramified at $p$, such that the quadratic twist $E^{\chi_p}$ attains good reduction at $p$.

\subsection{Infinitude of the set $\mathcal{P}(E)$} \label{sec:newprimes}

In this section, we prove that $\mathcal P(E)$ is infinite by treating the two alternatives in \Cref{prop:nonempty} separately. With $\mathcal{A}$ chosen as in \eqref{eq:calA-1}, if part~(1) holds for a set of auxiliary primes of upper density one, \Cref{lem:ss-conj} allows us to prove that $\mathcal P^{\mathrm{ss}}(E)$ is infinite (\Cref{prop:supersingular-infinite}). If part~(2) holds for infinitely many auxiliary primes, we use Hilbert class polynomials to construct new ordinary primes, proving that $\mathcal P^{\mathrm{ord}}(E)$ is infinite (\Cref{prop:ord-set-infinite}). Combining these results with \Cref{prop:nonempty}, we conclude that $\mathcal P(E)$ is infinite (\Cref{thm:qualitative-infinite}).

The proof of the following result proceeds as follows. We first choose a prime $p$ dividing a suitable value of a Hilbert class polynomial. Using assumptions \eqref{assump:Delta} and \eqref{assump:Kodaira}, we then construct a quadratic twist $E'$ of $E$ such that $p\neq 2$ and $E'$ has good reduction at $p$. Finally, according to whether $E'$ has supersingular or ordinary reduction at $p$, we show that either $p\in\mathcal{P}^{\mathrm{ss}}(E')$ or $p\in\mathcal{P}^{\mathrm{ord}}(E')$.


\begin{proposition}\label{prop:nonempty} 
Let $\mathcal A\subseteq\mathbb Z_{\geq1}$ contain $1$.
Then there exists a quadratic twist $E'$ of $E$ such that the set $\mathcal{P}(E')$ is nonempty. 

Furthermore,  for any 
$0\le \delta\le  1$, one of the following alternatives holds:
\begin{enumerate}
  \item The set \(\mathcal{P}^{\mathrm{ss}}(E')\) is nonempty, and the set of
    primes \(\ell\), satisfying
    \[
    \emptyset\neq
    \{p:p\mid \Num(P_{a^2\ell^2\Delta^*}(j_E))\}
    \subseteq \mathcal{P}^{\mathrm{ss}}(E') \quad \text{ for all $a \in \mathcal{A}\cap [1, \ell)$}
    \]
    is infinite and has upper density at least \(1-\delta\) among the primes.

    \item The set \(\mathcal{P}^{\mathrm{ord}}(E')\) is nonempty, and the set of primes 
    \(\ell\), satisfying
    \[
    \{p:p\mid \Num(P_{a^2 \ell^2\Delta^*}(j_E))\}
    \cap\mathcal{P}^{\mathrm{ord}}(E')\neq\emptyset \quad  \text{ for some $a\in \mathcal{A}\cap [1, \ell)$}
    \]
    is infinite and has lower density at least \(\delta\) among the primes.
\end{enumerate}
\end{proposition}

\begin{remark}\label{rem:twist-invariance} There may be several quadratic twists of $E$ satisfying the conclusion of the proposition above. The conclusion of \Cref{prop:nonempty} is invariant under quadratic twist. More precisely, if $E_1$ is any quadratic twist of $E$, then the possible curves $E'$ appearing in \Cref{prop:nonempty} remain the same. 
\end{remark}

\begin{proof}
By \Cref{lem:notpm1}, for any positive integer $f>c'(E, \Delta)$, there exists a prime $p$ dividing $\Num(P_{f^2\Delta^*}(j_E))$. We claim that $p$ is an odd prime by proving $2\nmid \Num(P_{f^2\Delta^*}(j_{E}))$. Recall that 
\[
|P_{f^2\Delta^*}(j_E)| := \prod_{1\leq k\leq h(f^2\Delta^*)}|j_E-j_k|,
\]
where for each $1\le k\leq h(f^2\Delta^*)$, $j_k$ is a singular modulus. Denote by  $E^{\text{CM}}_k$ the corresponding elliptic curve. Since $E^{\text{CM}}_k$ has potentially good reduction, we may assume $E^{\text{CM}}_k$ is defined over a number field \(L\) containing the ring class field \(L_{f^2\Delta^*}\) of $\mathcal{O}_{f^2\Delta^*}$. Let $\fp_k$ be a prime in $L$ above $2$, that is a prime of good ordinary reduction for $E^{\text{CM}}_k$. Our assumption that
  \[
 \Delta \equiv -1\pmod 8
  \]
  implies $2$ splits in the CM field $K:=\Q(\sqrt{-\Delta})$. Using Deuring's criterion \cite[\S 13, Theorem 12]{MR890960},  we see that  
$\fp_k$ is an ordinary prime for $E^{\text{CM}}_k$.  
Since the only supersingular $j$-invariant over $\overline{\F}_{2}$ is 0, we conclude that $j_k\not\equiv 0\pmod{\fp_k}$. 
 Since $2\mid \Num(j_E)$, it follows that $j_E-j_k\not\equiv 0\pmod {\fp_k}$ for any prime $\fp_k$ above 2. Note that   
\begin{align*}
N_{L_{f^2\Delta^*}/\Q}(j_{E}-j_k)&=N_{K/\Q}(N_{L_{f^2\Delta^*}/K}(j_{E}-j_k))\\
& =N_{K/\Q}(P_{f^2\Delta^*}(j_E))=P_{f^2\Delta^*}(j_{E})^2,
\end{align*}
where we use the fact that   $L_{f^2\Delta^*}= K\Q(j_k)$, which follows from CM theory.  This implies $2\nmid \Num(P_{f^2\Delta^*}(j_{E})^2)$.

Recall from the discussion at the beginning of the section, for any $p\mid N_E$, we can find a global character $\chi$ \cite[Theorem 5, p.103]{MR223335}, ramified only at primes dividing $2p$, such that $p\nmid N_{E^\chi}$. Repeat this process so that the final twisted elliptic curve $E'$ has good reduction at every odd bad prime.
Hence, every prime $p\mid P_{f^2\Delta^*}(j_{E'})$ is an odd prime of good reduction for $E'$. 
 
 If $p$ is a prime of supersingular reduction, then  \Cref{lem:deuring1} gives 
 $p\in \mathcal{P}^{\mathrm{ss}}(E')$. Hence it is nonempty.
On the other hand, if $p$ is a prime of ordinary reduction, then
\[
\End_{\overline{\F}_p}(E'_p)
\simeq \mathcal{O}_{f_0^2\Delta^*}
\]
for some $f_0\mid f$. Consequently,
\[
\End_{\overline{\F}_p}(E'_p)\otimes_{\Z}\Q
=
\Q(\pi_p(E'_p))
\simeq \Q(\sqrt{-\Delta}),
\]
and hence $\mathcal{P}^{\mathrm{ord}}(E')$ is nonempty.

If $f=\ell$ is a prime and
$\ell>c(E,\Delta^*)$, then, for every $a\in\mathcal A$,
$P_{a^2\ell^2\Delta^*}(j_E)\neq \pm 1.
$
In particular, for every \(a\in\mathcal A\),
\[
\{p: p\mid P_{a^2\ell^2\Delta^*}(j_E)\}\neq\varnothing.
\]

Now define
\[
\mathcal F_{\mathrm{ss}}(E') := \left\{ \ell\text{ prime}: \{p:p\mid P_{a^2\ell^2\Delta^*}(j_{E'})\} \subseteq\mathcal P^{\mathrm{ss}}(E') \text{ for every }a\in\mathcal A\cap[1, \ell) \right\}.
\]
Since $\mathcal{P}(E')=\mathcal{P}^{\mathrm{ss}}(E')\sqcup  \mathcal{P}^{\mathrm{ord}}(E')$ and every prime divisor of $P_{f^2\Delta^*}(j_E)$ belongs to $\mathcal{P}(E')$, the complement of $\mathcal F_{\mathrm{ss}}(E')$ among sufficiently large primes is
\begin{equation}\label{eq:FordE'}
\mathcal F_{\mathrm{ord}}(E') := \left\{ \ell\text{ prime}: \{p:p\mid P_{a^2\ell^2\Delta^*}(j_{E'})\} \cap\mathcal P^{\mathrm{ord}}(E')\neq\varnothing \text{ for some }a\in\mathcal A\cap [1, \ell)\right\}. 
\end{equation}
and therefore
$\underline d( \mathcal{F}_{\mathrm{ord}}(E')) = 1-\overline d( \mathcal F_{\mathrm{ss}}(E'))$, where $\underline d(\cdot)$ denotes the lower density and $\overline d(\cdot)$ denotes the upper density.
The rest of the claim in the proposition follows.
\end{proof}

\begin{remark}
     We cannot apply \cite[Theorem 3.10]{MR4929976} directly to produce a prime $p$ of good ordinary reduction of $E'$, since in our case  the CM field $\Q(\sqrt{-\Delta})$ is fixed.  
     \end{remark}

By \Cref{prop:nonempty}, after replacing \(E\) by a quadratic twist $E'$ if necessary, we may assume that \(\mathcal{P}(E')\) is nonempty. To prove that \(\mathcal{P}(E')\) is infinite, it suffices to show that
at least one of these two sets is infinite. We begin with the
supersingular case.  

\vspace{5pt}

\subsubsection{Supersingular Primes}\label{sec:ss-infinite}

Assume that part~(1) of \Cref{prop:nonempty} holds for primes $\ell$ of upper density $1-\delta=1$ and  
\begin{equation}\label{eq:calA-1}
    \mathcal{A}:=\{1\}\cup \{\prod_{1\leq i\leq k} p_i^{m_i}: p_i\text{ is a supersingular prime of $E'$, $k, m_i\in \Z_{\geq 1}$}\}.
\end{equation}
 Let \(\mathcal{F}_{ss}(E')\) denote the set of
primes \(\ell\) occurring in part~(1) of \Cref{prop:nonempty}, and put
\[
\mathcal{S}_\ell(E')
:=
\bigcup_{a\in \mathcal{A}\cap [1, \ell)}\{p:p\mid \Num(P_{a^2\ell^2\Delta^*}(j_{E'}))\}.
\]
Then, 
\[
\emptyset\neq\mathcal{S}_\ell(E')
\subseteq\mathcal{P}^{\mathrm{ss}}(E')
\]
for every \(\ell\in\mathcal{F}_{ss}(E')\).
Suppose the set 
\begin{equation}\label{eq:allsupersingular}
  \mathcal{S}(E'):=  \bigcup_{\ell\in \mathcal{F}_{ss}(E')} \mathcal{S}_\ell(E')
\end{equation}
is finite, and write it as $\{p_1, \ldots, p_k\}$. 
For each $1\leq i\leq k$, we recall that  there is an integer $m_{E'}(p_i, \Delta^*)$ in \Cref{lem:ss-conj} associated to  $p_i$, $D:=\Delta^*$, and $E:=E'$. Let 
\[
f_i:= \ell \prod_{\substack{j\neq i\\ 1\leq j\leq k}}  p_j^{m_j/2}, 
\]
where $m_j\geq m_{E'}(p_j, \Delta^*)$ is a fixed even integer, and $\ell$ is chosen so that $f_i$ is sufficiently large and coprime to $p_i$.
There is a density one  set of primes $\ell$ to choose from.  Put 
\[
f_0:= \prod_{\substack{1\leq i\leq k}}  p_i^{m_i/2}.
\]
Since for every $1\leq i\leq k$,
\[
(f_0\ell)^2\Delta^*
=
p_i^{m_i}
\left(\ell\prod_{j\ne i}p_j^{m_j/2}\right)^2\Delta^*,
\]
we apply \Cref{lem:ss-conj} at  $p_i$, $E'$, $\Delta^*$ with $m=m_i$ and $f=f_i$ and obtain
\begin{equation}\label{eq:singular-moduli-ss}
\frac{1}{ h(\ell^2f_0^2\Delta^*)}\sum_{1\le k\le h(\ell^2f_0^2\Delta^*)}
-\log |j_E-j_k|_{p_i}
\leq
C(E',p_i,\Delta^*).
\end{equation}

Note that part~(1) of \Cref{prop:nonempty} holds for upper density one set of primes, and  for $f_0<\ell$ sufficiently large, we have $f_0\in \mathcal{A}$. They imply there is upper density one set of primes $\ell$, such that  $|\Num(P_{\ell^2f_0^2\Delta^*}(j_{E'}))|$ is divisible  only by primes in $\{p_1, \ldots, p_k\}$. Since $j_k$ are $p$-integral and $p$ is a prime of good reduction of $E$,  
\begin{align}\label{eq:ss-quantitative}
\log |\Num(P_{\ell^2f_0^2\Delta^*}(j_{E'}))| &=-\sum_{p\in \mathcal{S}_{\ell}(E')} \log |\Num(P_{\ell^2f_0^2\Delta^*}(j_{E'}))|_p  \nonumber \\
& \leq \sum_{p\in \mathcal{S}_\ell(E')} C(E', p, \Delta^*)h(\ell^2f_0^2\Delta^*) \nonumber \\
&\ll_{E, \Delta^*, p_1, \ldots, p_k} k  \ell 
\end{align}
 for sufficiently large $\ell$.

On the other hand, from the proof of  \Cref{lem:notpm1} and \eqref{eq:class-number}, we have 
\[
\log |\Num(P_{\ell^2f_0^2\Delta^*}(j_{E'}))|\gg_{ E, \Delta^*, p_1, \ldots, p_k} \ell\log \ell.
\]
Comparing with its lower bound, we obtain
\[
k\gg_{E, \Delta^*, p_1, \ldots, p_k} \log \ell.
\]
Since we can find infinitely many primes $\ell$,  there is a contradiction since $k$ is a fixed finite number. This contradiction implies that  $\mathcal{S}(E')$ must be infinite. 

Finally, since
\[
\mathcal{S}(E')\subseteq\mathcal{P}^{\mathrm{ss}}(E')\subseteq \mathcal{P}^{\mathrm{ss}}(E)\cup \{p: p\mid N_E\},
\]
the infinitude of \(\mathcal{S}(E')\) implies the
infinitude of \(\mathcal{P}^{\mathrm{ss}}(E)\).

\begin{proposition}\label{prop:supersingular-infinite}
Assume  part~(1) of
\Cref{prop:nonempty} holds for primes $\ell$ of upper density one. Then \(\mathcal{P}^{\mathrm{ss}}(E)\) is infinite.
\end{proposition}

\vspace{5pt}

\subsubsection{Ordinary Primes}\label{subsec:ord}
Assume that the set of auxiliary primes satisfying part~(2) of \Cref{prop:nonempty} is infinite, with $\mathcal{A}$ as in \eqref{eq:calA-1}. 
For any $E'$ given in \Cref{prop:nonempty}, the set $\mathcal{P}^{\mathrm{ord}}(E')$ is nonempty. 
Let $\Sigma_0(E')$ be a fixed finite nonempty subset of $\mathcal{P}^{\mathrm{ord}}(E')$ and let $k:=\#\Sigma_0(E')$.

To construct a new prime in $\mathcal{P}^{\mathrm{ord}}\backslash \Sigma_0(E')$, we choose an integer $L_k\ge 3$ satisfying:
\begin{enumerate}
\item \label{eq:condition 0} $ P_{L_k^2\Delta^*}(j_E)\neq \pm 1$.
    \item \label{eq:condition 1} $p\nmid P_{L_k^2\Delta^*}(j_E)$ for all $p\in \Sigma_0(E')$. 
\end{enumerate}

\begin{definition}
    We call an integer \textit{admissible} (with respect to $\Sigma_0(E')$)  if it satisfies  conditions  \eqref{eq:condition 0} and \eqref{eq:condition 1}.
\end{definition}

Denote 
 \[
K_0:= \max_{p\in \Sigma_0(E')}\{p, c(E, \Delta)\},
\]
where $c(E, \Delta)$ is given by \Cref{lem:notpm1}. 
\begin{lemma}\label{prop:admissible}
   There exists an absolute constant $C_0\ge 1$ such that $\ell a$ is admissible for every prime $\ell > C_0 K_0$ and $a\in \mathcal{A}$.
\end{lemma}
\begin{proof}
Let $p_1,\dots,p_k$ denote the primes in $\Sigma_0(E')$, and write  $\mathcal{O}_i:=\End_{\overline{\F}_{p_i}}(E'_{p_i})$. This is an order with conductor $f_i\geq 1$ and discriminant  $-f_i^2\Delta^*$. Take an integer $L_k\geq c(E, \Delta)$ that is not equal to one of $f_1, \cdots, f_{k}$, and coprime to $p_1, \ldots, p_k$. 
We claim that such $L_k$ is admissible. 

 By  \Cref{lem:notpm1}, we have $P_{L_k^2\Delta^*}(j_E)\neq \pm 1$, and hence condition \eqref{eq:condition 0} holds. Now we show condition \eqref{eq:condition 1} holds. Suppose $p_i\mid P_{L_k^2\Delta^*}(j_E)$ for some $i\in \{1, \ldots, k\}$. Since  $p_i\in \Sigma_0(E')$, $p_i$ is a prime of good ordinary  reduction of $E'$. By \Cref{lem:deuring2} and the assumption that  $p_i\nmid L_k\Delta^*$, \footnote{Note that \(p_i\) is a prime of good ordinary reduction and hence splits in \(\Q(\sqrt{-\Delta})\). It follows that \(p_i\) cannot divide the discriminant \(\Delta^*\).}  there is  an isomorphism  $\End_{\overline{\F}_{p_i}}(E'_{p_i}) \simeq \mathcal{O}_{L_k^2\Delta^*}$ and hence  $ \mathcal{O}_{L_k^2\Delta^*}\simeq \mathcal{O}_{f_i^2\Delta^*}$. Since quadratic orders are uniquely determined by their discriminants, it follows that $
L_k^2\Delta^*=f_i^2\Delta^*$. 
Thus $f_i = L_k$, contradicting $L_k\neq f_i$. It follows that none of the primes \(p_1,\ldots,p_k\) divides \(P_{L_k^2\Delta^*}(j_E)\).

Finally, if $L_k$ is of the form $\ell a$ for some $a\in \mathcal{A}$, then $L_k$ is admissible if    
\[
\ell>\max_{1\le i\leq k}\{p_i, f_i,  c(E, \Delta)\}.
\]
The inclusion of orders given in \eqref{eq:inclusion-orders} gives
\begin{equation}\label{eq:conductor-bound}
f_i^2\Delta^*= |\disc (\End_{\overline{\F}_{p_i}}(E'_{p_i}))|\leq |\disc\Z[\pi_{p_i}(E')]|=|a_{p_i}(E')^2-4p_i|\leq 4p_i.
\end{equation}
Therefore, $f_i\leq 2\sqrt{p_i}$, and hence there exists an absolute constant $C_0>0$, such that $\ell a$ is admissible for every prime $\ell$ satisfying
\[
\ell\geq C_0 \max_{1\le i\leq k}\{p_i,   c(E, \Delta)\}
\]
and every $a\in \mathcal{A}$.
\end{proof}

For each prime $\ell>C_0K_0$ satisfying condition (2) of \Cref{prop:nonempty}, choose $a(\ell)\in\mathcal A$ such that $\Num(P_{a(\ell)^2\ell^2\Delta^*}(j_E))$ has a prime divisor in $\mathcal P^{\mathrm{ord}}(E')$. Take  $L_\ell:=\ell a(\ell)$ as our admissible integer with respect to $\Sigma_0(E')$.
Since $\ell$ satisfies condition \eqref{eq:condition 0}, the integer $\Num(P_{L_\ell^2\Delta^*}(j_E))$ has  a prime divisor. Choose one and denote it by
\begin{equation}\label{eq:definep_k+1}
p_{k+1}:=p_{k+1}(L_\ell). 
\end{equation}
Now by condition \eqref{eq:condition 1}, we also have $p_{k+1}\notin \Sigma_0(E')$. 
We will show that \(p_{k+1}\) can be chosen so that \(p_{k+1}\in \mathcal{P}^{\mathrm{ord}}(E')\). 

\begin{lemma}\label{lem:new-prime}
There exist an admissible integer $L_\ell$  with respect to $\Sigma_0(E')$ and a choice of $p_{k+1}(L_\ell)$, such that 
\[
\mathcal{P}^{\mathrm{ord}}(E')\supseteq \Sigma_0(E')\sqcup  \{p_{k+1}(L_\ell)\}.
\]
\end{lemma}
\begin{proof}
 By \Cref{prop:admissible}, for every sufficiently large prime $\ell$, $L_\ell$ is admissible. Hence, from the assumption in part (2) of \Cref{prop:nonempty},  we can choose an admissible integer $L_\ell$ and a new prime $p_{k+1}$ so that $p_{k+1}$ is a prime of good ordinary reduction  for $E'$. Since $p_{k+1}\mid \Num(P_{L_\ell^2\Delta^*}(j_{E'}))$,  by \Cref{lem:deuring1}, there is an isomorphism $\End_{\overline{\F}_p}(E'_p)\otimes_\Z \Q\simeq \Q(\sqrt{-\Delta})$. Therefore, $p_{k+1}\in \mathcal{P}^{\mathrm{ord}}(E')$. 
 This completes the proof.
\end{proof}

We call any $p_{k+1}$ given in \Cref{lem:new-prime} a \textit{new prime} (associated to the admissible integer $L_\ell$) in $\mathcal{P}^{\mathrm{ord}}(E')$. 
We emphasize that it is not necessarily true that $p_{k+1}>\max\{p: p\in \Sigma_0(E')\}$.

\begin{proposition}\label{prop:ord-set-infinite}
Under the assumptions stated at the beginning of \Cref{subsec:ord}, the set  $\mathcal{P}^{\mathrm{ord}}(E)$ is infinite. 
\end{proposition}
\begin{proof}
    As noted in \Cref{sec:setup}, to show $\mathcal{P}^{\mathrm{ord}}(E)$ is infinite, it suffices to show $\mathcal{P}^{\mathrm{ord}}(E')$ is infinite for a quadratic twist $E'$ of $E$.  

   By part (2) of   \Cref{prop:nonempty},  we  assume the set $\mathcal{P}^{\mathrm{ord}}(E')$ is nonempty.
   Suppose that $\mathcal{P}^{\mathrm{ord}}(E')$ is finite. Applying  \Cref{lem:new-prime} with $\Sigma_0(E')=\mathcal{P}^{\mathrm{ord}}(E')$, we obtain that $\#\mathcal{P}^{\mathrm{ord}}(E')\geq \#\Sigma_0(E')+1$, a contradiction. 
  Therefore,  $\mathcal{P}^{\mathrm{ord}}(E')$, and hence  $\mathcal{P}^{\mathrm{ord}}(E)$ is infinite. 
\end{proof}
\vspace{5pt}

Combining \Cref{prop:nonempty}, \Cref{prop:supersingular-infinite}, and  \Cref{prop:ord-set-infinite}, we immediately obtain that $\mathcal{P}(E)$ is infinite. 

\begin{theorem}\label{thm:qualitative-infinite}
 Let $E$ and $\Delta$ be the same as in \Cref{thm:lower-bound}. Then $\mathcal{P}(E)$ is infinite. 
\end{theorem}

\subsection{Quantitative bounds}\label{sec:quantitative-lower}

In this section, we give quantitative lower bounds of supersingular and ordinary primes counted in $\tilde{\pi}_E(x, \Q(\sqrt{-\Delta}))$ separately. In the supersingular case, we need to assume the uniform version of \Cref{conj:super}. In the ordinary case, we derive a recursive bound between the newly obtained primes and those arising at earlier stages.

We fix a twist $E'/\Q$ and a density parameter $0<\delta<1$  in \Cref{prop:nonempty}, and take $\mathcal{A}$ as defined in \eqref{eq:calA-1}. We conclude the section by proving \Cref{thm:lower-bound}.

\subsubsection{Supersingular primes}\label{sec:quantitative-lower-ss}

Recall $\mu$ from \Cref{conj:super}. To prove the following result, we only use the fact that $1\in \mathcal{A}$.

\begin{proposition}\label{thm:lower-bound-ss}
 Assume part (1) of \Cref{prop:nonempty} and the uniform version of \Cref{conj:super} holds for primes $\ell$ of upper density at least $1-\delta>0$. Then, there are infinitely many $x$ such that 
 \[
  \tilde{\pi}^{\text{ss}}_E(x, \Q(\sqrt{-\Delta}))\gg_{E,  \Delta^*} (\log x)^{1-\mu}
 \]
\end{proposition}
\begin{proof}
 By part (1) of \Cref{prop:nonempty}, there are infinitely many $X$ such that at least $\gg X/\log X$ primes satisfy $\ell\in \mathcal{F}_{ss}(E')\cap [X, f_0(X)]$, where  $f_0(X):=\frac{100}{1-\delta}X$. For each such prime $\ell$,  every prime divisor $p\mid P_{\ell^2\Delta^*}(j_E)$ is a supersingular prime.   
 By \Cref{lem:height},  there exists a constant $\gamma_{E, \Delta}$ such that  any such prime  $p$ satisfies 
  \[
  p\leq g(\ell):=\exp(\gamma_{E, \Delta} \ell(\log \ell)^2).
  \]

Let $\mathcal{L}(E', X):=\mathcal{F}_{ss}(E')\cap[X, f_0(X)]$. 
If for at least half of the primes $\ell\in \mathcal{L}(E', X)$, there is a prime divisor $p\mid P_{\ell^2\Delta^*}(j_{E'})$ such that $p\nmid  P_{\ell'^2\Delta^*}(j_E)$ for any $\ell'\in \mathcal{L}(E', X)$ with $\ell\neq \ell'$, then, we have 
\[
\tilde{\pi}_{E'}^{ss}(g(f_0(X)), \Q(\sqrt{-\Delta}))\gg X/\log X.
\]
By change of variable, we obtain that for infinitely many $x$, 
\[
\tilde{\pi}_{E'}^{ss}(x, \Q(\sqrt{-\Delta}))\gg  \frac{\log x}{(\log \log x)^3}.
\]

Otherwise, there are more than
half of the primes  $\ell\in \mathcal{L}(E', X)$, such that every prime divisor $p\mid P_{\ell^2\Delta^*}(j_{E'})$ also divides $P_{\ell'^2\Delta^*}(j_{E'})$ for some $\ell'\in \mathcal{L}(E', X)$ and $\ell'\ne \ell$. Denote the set of such primes $\ell$ by $\mathcal{F}'_{ss}(X)$.  By \cite[Theorem 2]{MR1040429}, there is a better bound for $p$:
\[
p\leq \ell\ell'\Delta^* \leq f_1(X):=\frac{10000}{(1-\delta)^2}\Delta^* X^2.
\]
Therefore, for any $\ell\in \mathcal{F}'_{ss}(X)$,  applying \eqref{eq:ss-quantitative},  \Cref{lem:notpm1}, \eqref{eq:class-number}, and \Cref{conj:super}, we get 
  \begin{align*}
  \ell \log \ell 
& \ll_{E, \Delta}  \log |P_{\ell^2\Delta^*}(j_{E'})|
\nonumber \\
&\ll_{E, \Delta } -   \sum_{\substack{p\leq g(\ell) \\ p\in \mathcal{S}_\ell(E')}} \log |P_{\ell^2\Delta^*}(j_{E'})|_p \nonumber \\
& \ll_{E, \Delta}    \sum_{\substack{p\leq f_1(X)\\ p\in \mathcal{S}(E')}}  h(\ell^2\Delta^*)  (\log p)^\mu\\
& \ll_{E, \Delta} \ell \sum_{\substack{p\leq f_1(X)\\ p\in \mathcal{S}(E')}}   (\log p)^\mu.
\end{align*}
Hence, 
\[
\tilde{\pi}_{E'}^{ss}(f_1(X), \Q(\sqrt{-\Delta}))(\log X)^\mu\gg_{E, \Delta} \log \ell\gg \log X.
\]
Since $0<\mu<1$, by change of variable, we obtain that for infinitely many $x$, 
\[
\tilde{\pi}_{E'}^{ss}(x, \Q(\sqrt{-\Delta}))\gg (\log x)^{1-\mu}.
\]
Since at least one of the two lower bounds occurs for infinitely many good \(X\), the result now follows. 
\end{proof}

\subsubsection{Ordinary primes}\label{sec:quantitative-lower-ord}
Assume part~(2) of \Cref{prop:nonempty} holds. In particular,  there exists a set of primes $\ell$ of lower density at least $\delta>0$, such that  $P_{L_\ell^2\Delta^*}(j_E)$ is divisible by an ordinary prime of $E'$. 

The following result shows that any three distinct admissible auxiliary primes produce at least two distinct new primes.
 
\begin{lemma}\label{lem:distinct-newprimes}
 Let $\ell, \ell', \ell''$ be distinct primes and assume $L_\ell, L_{\ell'}, L_{\ell''}$ are three admissible integers with respect to $\Sigma_0(E')$, each of which generates a new prime. Then we can find at least two distinct new primes among $\{p_{k+1}(L_\ell),  p_{k+1}(L_{\ell'}), p_{k+1}(L_{\ell''})\}$. 
\end{lemma}
\begin{proof}
Suppose, for a contradiction, that all three chosen new primes
coincide, and denote their common value by
\[
p:=p_{k+1}(L_\ell)
  =p_{k+1}(L_{\ell'})
  =p_{k+1}(L_{\ell''})\in\mathcal E(E').
\]
For each $r\in\{\ell,\ell',\ell''\}$, we have
\[
p\mid \Num(P_{L_r^2\Delta^*}(j_{E'})).
\]
Since $p$ is ordinary and every prime divisor of $a(r)$
is supersingular for $E'$, we have $p\nmid a(r)$.
Thus $p\mid L_r=r\,a(r)$ implies $p=r$.
As $\ell,\ell',\ell''$ are pairwise distinct, $p$ divides
at most one of $L_\ell,L_{\ell'},L_{\ell''}$.

Consequently, there are distinct
$r,s\in\{\ell,\ell',\ell''\}$ such that $p\nmid L_rL_s$.
Moreover, $p\nmid\Delta^*$, since $p\in\mathcal E(E')$.
Applying \Cref{lem:deuring2} twice gives
\[
\End_{\overline{\F}_p}(E'_p)
 \simeq\mathcal O_{L_r^2\Delta^*}
 \simeq\mathcal O_{L_s^2\Delta^*}.
\]
Equality of the discriminants therefore implies $L_r=L_s$.
But $a(r)<r$ and $a(s)<s$, so $r$ and $s$ are respectively
the largest prime divisors of $L_r$ and $L_s$.
Hence $r=s$, a contradiction.
\end{proof}

Now we are ready to give the lower bound for ordinary primes.   

Fix a finite set of primes $\Sigma_0(E')$. Recall from \eqref{eq:FordE'} that 
\[
\mathcal F_{\mathrm{ord}}(E') := \left\{ \ell\text{ prime}: \{p:p\mid \Num(P_{L_\ell^2\Delta^*}(j_{E'}))\} \cap\mathcal E(E')\neq\varnothing \right\}. 
\]
By the lower density assumption (2) in \Cref{prop:nonempty} and \Cref{prop:admissible}, we have that 
 \[
\#\{\ell\le L:\ell\in\mathcal F_{\mathrm{ord}}(E'),
\ L_\ell\text{ admissible}\}
\gg \frac{L}{\log L}.
\]
For every such $\ell$, choose an ordinary prime divisor
\[
p(L_\ell)\mid \Num(P_{L_\ell^2\Delta^*}(j_{E'})).
\]
Since $L_\ell$ is admissible, we have $p(\ell)\notin\Sigma_0(E')$. Applying \Cref{lem:height},  \Cref{prop:admissible}, and using 
\begin{equation}\label{eq:a-ll-upper-bound}
a(\ell)<\ell,
\end{equation}
there is a constant $C>0$ such that 
\[
\log p(\ell)\leq \log |P_{L_\ell^2\Delta^*}(j_E)|\leq C\ell^2(\log \ell)^2\leq CL^2(\log L)^2.
\]
Moreover, by \Cref{lem:distinct-newprimes}, any three distinct choices of such primes $\ell$ yield at least two distinct new primes.
Therefore, 
\[
 \tilde{\pi}_{E'}^{\mathrm{ord}}(\exp(CL^2(\log L)^2), \Q(\sqrt{-\Delta})) \gg_{E, \Delta} \frac{L}{\log L}.
\]
Finally, by change of variable, we get 
\[
 \tilde{\pi}_{E'}^{\mathrm{ord}}(x, \Q(\sqrt{-\Delta})) \gg_{E, \Delta} \frac{(\log x)^{1/2}}{(\log\log x)^2}.
\]
Since \[
\tilde{\pi}^{\mathrm{ord}}_{E}(x,\mathbb{Q}(\sqrt{-\Delta}))
\geq
\tilde{\pi}^{\mathrm{ord}}_{E'}(x,\mathbb{Q}(\sqrt{-\Delta}))
-\omega(N_E),
\]
we immediately obtain the following. 
\begin{proposition}\label{thm:lower-bound-ord}
   Assume part (2) of \Cref{prop:nonempty} holds for primes $\ell$ of lower density at least $\delta>0$. Then, 
   \[  \tilde{\pi}_E^{\mathrm{ord}}(x, \Q(\sqrt{-\Delta}))  \gg_{E, \Delta}\frac{(\log x)^{1/2}}{(\log\log x)^2}.
   \]
\end{proposition}

Combining the supersingular and ordinary cases, we obtain the following quantitative result. 
\begin{theorem}\label{thm:quantitative}
Let $E$ and $\Delta$ be the same as in \Cref{thm:lower-bound}. Assume the uniform version of \Cref{conj:super}.  Then either $\tilde{\pi}_E(x, \Q(\sqrt{-\Delta}))\gg_{E, \Delta} (\log x)^{1-\mu}$ for infinitely many $x$ or $\tilde{\pi}_E(x, \Q(\sqrt{-\Delta}))\gg_{E, \Delta} \frac{(\log x)^{1/2}}{(\log\log x)^2}.$ 
\end{theorem}

To conclude this section, we complete the proof of  \Cref{thm:lower-bound}.
\begin{proof}[Proof of \Cref{thm:lower-bound}]
 Fix a twist $E'$ with good reduction at every odd prime. Suppose that statement (2) of \Cref{thm:lower-bound} does not hold; that is, suppose that the set $\mathcal{P}^{\mathrm{ss}}(E')$ is finite. 
 
 Write
$
\mathcal P^{ss}(E')=\{p_1,\ldots,p_r\}.
$
Choose fixed even integers
$
m_i\ge m_{E'}(p_i,\Delta^*),
$
and define 
\[
f_0=\prod_{p_j\in \mathcal{P}^{\mathrm{ss}}(E')}p_j^{m_j/2}.
\] 
Choose $\mathcal{A}=\{1, f_0\}$
and fix an appropriate density parameter $\delta$ in \Cref{prop:nonempty}. The argument after \eqref{eq:allsupersingular} in \Cref{sec:ss-infinite}  remains valid. 
 Hence condition (1) of \Cref{prop:nonempty} cannot hold, and we are therefore in the setting where part (2) of \Cref{prop:nonempty} holds. Since the set $\mathcal{A}$ is finite, we can replace \eqref{eq:a-ll-upper-bound} by $a(\ell)\in \{1, f_0\}$. Following the proof above, there is a constant $C>0$ such that 
\[
\log p(\ell)\leq \log |P_{L_\ell^2\Delta^*}(j_E)|\leq C\ell(\log \ell)^2\leq CL(\log L)^2.
\] 
Then, we get
 \[
\tilde{\pi}_{E'}^{\mathrm{ord}}\left(\exp(CL(\log L)^2),\Q(\sqrt{-\Delta})\right)\gg_{E, \Delta} \frac L{\log L}.
\]
Hence, 
 \[
  \tilde{\pi}_{E}^{\mathrm{ord}}(x, \Q(\sqrt{-\Delta})) \gg_{E, \Delta} \frac{\log x}{(\log\log x)^3}.
 \]
 The desired lower bound then follows from  \eqref{eq:compare-with-tilde}. 
\end{proof}

\section{The upper bound}\label{sec:proof-upper}

We keep the notation from \Cref{sec:preliminaries} and \Cref{sec:Theorem2}.  
Let $E/\Q$ be a non-CM elliptic curve and $\Delta\geq 1$ be a squarefree integer as given in \Cref{thm:upper-bound}. We will give upper bounds for  $\tilde{\pi}_E(x, \Q(\sqrt{-\Delta}); \delta)$  for any $0<\delta<1/2$. 

\subsection{An easy upper bound}\label{sec:easy-upper-bound}

In this section, we derive a ``trivial'' unconditional upper bound for the counting function $\tilde{\pi}_E(x, \Q(\sqrt{-\Delta}); \delta)$.

Let $p$ be an ordinary prime counted by $\tilde{\pi}_E(x, \Q(\sqrt{-\Delta}))$. Then  we have 
$
\End_{\overline{\F}_p}(E_p)\simeq \mathcal{O}_{f_p^2\Delta^*},
$
where
\[
f_p:=[\mathcal{O}_{\Q(\sqrt{-\Delta})}: \End_{\overline{\F}_p}(E_p)].
\]
In particular, $p$ is also counted by $\pi_E(x, \Q(\sqrt{-\Delta}))$. 
From the Hasse bound for $a_p(E)$, we derive that  $1\le f_p\leq 2\sqrt{p}$ (see also  \eqref{eq:conductor-bound}). 
Therefore, by \Cref{lem:deuring2},  
every such prime $p\geq 5$  divides $\Num(P_{f_p^2\Delta^*}(j_E))$ for some  $1\le f_p\leq 2x^{1/2}$. 
If $\delta=1/2$, then we have the relations 
\[
\pi_E(x, \Q(\sqrt{-\Delta}))\leq  \tilde{\pi}_E(x, \Q(\sqrt{-\Delta}); 1/2)+O(1)\le  \tilde{\pi}_E(x, \Q(\sqrt{-\Delta}))+O(1).
\]

If $p$ is a supersingular prime, and is counted by $\tilde{\pi}_E(x, \Q(\sqrt{-\Delta}); \delta)$, then $E_p$ has CM by $\mathcal{O}_{f_p'^2\Delta^*}$ for some $f'_p\mid f_p$. In particular, $f'_p\leq 2x^\delta$. Hence, by Deuring's lifting theorem, we get $p\mid \Num(P_{f_p'^2\Delta^*}(j_E))$.  

For any $0<\delta<1/2$ and a fundamental discriminant $-D$ with $D>0$, 
we denote  
\begin{equation}\label{eq:G-delta-D}
G_{\delta}(x; D):=\#\{p\leq x: p\mid \Num(P_{f^2D}(j_E)) \text{ for some $f\in [1, 2x^{\delta}]$}\}.
\end{equation}
If $D=\Delta^*$, we simply write 
\[
G_{\delta}(x):=G_{\delta}(x; \Delta^*).
\]
From the discussion above, we see that 
\begin{equation}\label{eq:thm2-G}
\tilde{\pi}_{E}(x, 
\Q(\sqrt{-\Delta}); \delta)\leq G_\delta(x)+O_E(1).
\end{equation}
We have the following ``trivial"  upper bound of  $G_\delta(x)$.

\begin{proposition}\label{lem:trivial} Let $E$ and $\Delta$ be as above. If $0<\delta<1/2$, then 
\[
G_\delta(x)\ll_{E}  \frac{\sqrt{\Delta}}{\delta}(\log (2\Delta)+ \log x) x^{2\delta}.
\]
\end{proposition}
\begin{proof}
    By \Cref{lem:height}, we have 
    \begin{align*}
        G_\delta(x) \leq  \sum_{1\leq f\leq 2x^\delta} \omega\left(\Num(P_{f^2\Delta^*}(j_E))\right) & \ll x^\delta\max_{1\leq f\leq 2x^\delta}\frac{\log \left |\Num(P_{f^2\Delta^*}(j_E))\right|}{\log\log\left|\Num(P_{f^2\Delta^*}(j_E))\right|}\\
        & \ll_{E} x^\delta\cdot  \frac{x^\delta\sqrt{\Delta}(\log (x\sqrt{\Delta}))^2}{\delta \log(x\sqrt{\Delta})}\\
        & \ll_E \frac{\sqrt{\Delta}}{\delta}(\log (3\Delta)+ \log x) x^{2\delta},
    \end{align*}
    where we use the bound  $\omega(n)\ll \frac{\log n}{\log\log n}$.  This completes the proof.
\end{proof}

In what follows, for each fixed $\Delta$, we obtain a  modest unconditional improvement on \Cref{lem:trivial} for sufficiently large $x$, using Arakelov intersection theory on the coarse moduli scheme  $X(1)$ over $\operatorname{Spec}\Z$ of generalized elliptic curves.  

\subsection{Proof of \Cref{thm:upper-bound}}\label{sec:proof-thm-upper}

Recall the notation from \Cref{sec:equidistribution}. 
In this section, we prove   \Cref{thm:upper-bound}. In fact, we will give a general result, \Cref{prop:general-bound-G}, that allows the conductor $f_p$ with $p\leq x$ to be in any sufficiently small subset of $[1, 2x^{\delta}]$.

We will apply \Cref{prop:Charles} by taking   $Z$ and $Y$ to be the horizontal divisors on $X(1)_\Z$ obtained by taking  the Zariski closures of the point  $j_E$ and the irreducible divisor over $\Q$
\begin{equation}\label{eq:heggner}
y_{\Delta^*}:=j_1+\ldots +j_{h(\Delta^*)},    
\end{equation}
inside $X(1)_{\Z}$, respectively, 
where $j_1, \ldots, j_{h(\Delta^*)}$ are the $j$-invariants of elliptic curves with CM by the maximal order  $\mathcal{O}_{\Q(\sqrt{-\Delta})}$. Note that $Y$ and $Z$ are defined over $\Q$ and have no common components since $E/\Q$ is a non-CM elliptic curve. 

Let $N>1$ be an integer. We recall a classical result on the Hecke orbit of CM points in order to describe $t_{N*}Y$. Let  $d>0$ be an integer such that $-d$ is a discriminant.  We denote by
\begin{equation}\label{eq:CMdivisor}
y_d:=j_1+\cdots +j_{h(d)}
\end{equation}
the corresponding CM divisor  on $X(1)_\C$, consisting of all  
CM $j$-invariants of discriminant $-d$.  For each fundamental discriminant $-d$, we also denote by $R_d:\mathbb{Z}_{\geq 1} \to \Z_{\ge 0}$ the function defined by
\[
R_d(n):=\#\{I \trianglelefteq \mathcal{O}_{\Q(\sqrt{-d})}: N(I)=n\}. 
\]
The image of $y_{\Delta^*}$ under the Hecke correspondence $T_N$ in the sense of \cite[Lemma 2.6]{MR2058609} (see also \cite[Proposition 4.2.1]{MR1826411} or \cite[Section 2C]{MR4129386})  is given by 
\begin{equation}\label{eq:Hecke-CM}
T_N(y_{\Delta^*})=\begin{cases}
R_{\Delta^*}(N) y_{\Delta^*}+R_{\Delta^*}(1)y_{N^2\Delta^*}+\displaystyle\sum_{\substack{1<f<N \\ f\mid N}} R_{\Delta^*}(N/f) y_{f^2\Delta^*}   & \text{if $\Delta^*\neq 3, 4$}\\
 R_{\Delta^*}(N) y_{\Delta^*}+3R_{\Delta^*}(1)y_{N^2\Delta^*}+3\displaystyle\sum_{\substack{1< f< N\\ f\mid N}}  R_{\Delta^*}(N/f) y_{f^2\Delta^*}       & \text{if $\Delta^*=3$}\\
R_{\Delta^*}(N) y_{\Delta^*}+2R_{\Delta^*}(1)y_{N^2\Delta^*}+2\displaystyle \sum_{\substack{1<f< N\\ f\mid N}}  R_{\Delta^*}(N/f) y_{f^2\Delta^*}        & \text{if $\Delta^*=4$}\\
\end{cases}.
\end{equation}
On the other hand, the \(N\)-th Hecke orbit \(t_{N*}\) defined in \Cref{sec:equidistribution} only involves cyclic isogenies. Therefore, by CM theory and the proof of \cite[Proposition 4.2.1]{MR1826411}, the expression for \(t_{N*}y_{\Delta^*}\) is analogous to \eqref{eq:Hecke-CM} with each coefficient \(R_{\Delta^*}(n)\) replaced by 
\[
r_{\Delta^*}(n):=\#\{I \trianglelefteq \mathcal{O}_{\Q(\sqrt{-\Delta})}: \mathcal{O}_{\Q(\sqrt{-\Delta})}/I \text{ is a cyclic group  and }N(I)=n\}.
\]
Moreover, we have $\deg(t_{N*}y_{\Delta^*})=e_N h(\Delta^*)$, where $e_N=N\prod_{p\mid N}(1+1/p)$. For each prime $p$ coprime to the conductor of $E$, we see the local arithmetic degree at $p$ is 
\begin{equation}\label{eq:local-degree}
\deg_p(Z.t_{N*}Y)=\sum_{\alpha\in |t_{N*}y_{\Delta^*}|} m_\alpha\max\{0,-\log |j_{\alpha}-j_E|_p\}, 
\end{equation}
and at the unique archimedean place,
\[
\log ||s_{Z}(t_{N*}Y) ||=\sum_{\alpha\in |t_{N*}y_{\Delta^*}|} m_\alpha\log\left(|j_\alpha-j_E| |\Delta(\tau_{\alpha}) |\Im(\tau_\alpha)^6 \right),
\]
where  $m_\alpha\geq 1$ is the multiplicity, $j_\alpha=j(\tau_\alpha)$,      $\tau_\alpha\in X(1)(\C)=\SL_2(\Z)\backslash \overline{\mathbb{H}}$ is in the fixed fundamental domain, and $\Im(\tau_\alpha)>0$ is the imaginary part of $\tau_\alpha$.

Note that $|\Delta(\tau_{\alpha}) |\Im(\tau_\alpha)^6$ is continuous and nonvanishing on $\SL_2(\Z)\backslash \mathbb{H}$, and from the Fourier expansion of $\Delta(\tau_{\alpha})$ near the cusp, this quantity decays exponentially as  $\Im (\tau_\alpha)\to \infty$. We now estimate this quantity in  terms of $j_\alpha$. 
Since $\Delta(\tau)=q+O(q^2)$, where $q=e^{2\pi i \tau}$,  we have  
\begin{equation}\label{eq:delta}
|\Delta(\tau)\Im (\tau)^6|=\Im(\tau)^6 e^{-2\pi \Im(\tau)}(1+o(1)), \quad \text{ as $\Im (\tau) \to \infty$}. 
\end{equation}
On the other hand, since 
$j(\tau)=q^{-1}+744+O(q),$
\begin{equation}\label{eq:j-invariant}
|j(\tau)|=e^{2\pi \Im(\tau)}(1+o(1)), \quad \text{as $\Im (\tau) \to \infty$}.
\end{equation}
Combining \eqref{eq:delta} and \eqref{eq:j-invariant}, we derive that as $\Im(\tau_\alpha) \to \infty$, 
\[
|j_\alpha-j_E||\Delta(\tau_\alpha)|\Im (\tau_\alpha)^6 \asymp |j_\alpha-j_E|\frac{(\log |j_{\alpha}|)^6}{|j_{\alpha}|} \asymp_E (\log |j_{\alpha}|)^6.
\]
Then by \cite[Lemma 2.1]{FoTr2026}, we have
\[
\log |j_{\alpha}|\ll \Im (\tau_\alpha) \ll N\sqrt{\Delta^{*}} 
\]
as $\Im(\tau_\alpha)\to \infty$.
Combining the above estimates, we obtain  that as $N\to \infty$,
\begin{equation}\label{eq:arch-estimation}
   \log ||s_{Z}(t_{N*}Y) || \ll_{E} e_N h(\Delta^*)\log (N\sqrt{\Delta^*}).  
\end{equation}

Next, we estimate the non-archimedean contribution. 
Fix a prime $p$ of good reduction for $E$ and the CM elliptic curves that is also counted by $G_\delta(x)$, and choose
$N=f_0(p)$ such that
\[
p\mid \operatorname{Num}\bigl(P_{N^2\Delta^*}(j_E)\bigr).
\]
Now by \eqref{eq:Hecke-CM},  \eqref{eq:local-degree}, and positivity of the local intersection number,  we obtain the lower bound for the local intersection number
at $p$:
\begin{align}\label{eq:non-arch-lower}
   \deg_p(Z.t_{N*}Y) &\geq -r_{\Delta^*}(1)\sum_{1\le i\le h(N^2\Delta^*)}\log |j_i-j_E|_p \nonumber \\
   & \geq -\log |P_{N^2\Delta^*}(j_E)|_p\geq  v_p(P_{N^2\Delta^*}(j_E))\log p\geq \log p.
\end{align}

Define
\begin{equation}\label{eq:F-delta}
\mathcal{F}_\delta:=\{f_0(p)\in [1, 2x^{\delta}]: p \text{ is counted by $G_{\delta}(x)$}\}. 
\end{equation}
Let $\eta$ be a real number such that $0<\eta\leq \delta$, and let $\mathcal{G}_\eta $ be any subset of $\mathcal{F}_\delta$ satisfying  
\[
\#\mathcal{G}_\eta\leq C_{\mathcal{G}_\eta}  x^{\eta}
\]
 for some constant $C_{\mathcal{G}_\eta}\in (0, 2]$.  Denote 
 \begin{equation}\label{eq:upper-anysubset}
G_\mathcal{G_\eta}(x):=\#\{x^{\delta}\le p\leq x: p\nmid N_E, p\mid \Num(P_{f^2\Delta^*}(j_E)) \text{ for some $f\in \mathcal{G}_\eta$}\}. 
\end{equation}
Note that 
\begin{equation}\label{eq:general-bound}
    G_\delta(x)=G_{\mathcal{F_\delta}}(x)+O(x^{\delta}). 
\end{equation}
We first give an upper bound for $G_{\mathcal{G_\eta}}(x)$ with $0<\eta\leq \delta$.

\begin{proposition}\label{prop:general-bound-G}
    Let $E$, $\Delta$, $\delta$ be as  above. Let $0<\eta\le \delta$.  Then, for all sufficiently large $x$,  we have 
        \[
        G_\mathcal{G_\eta}(x)\ll_E \Delta^{1/2}(\log (3\Delta))^2 x^{\delta+\eta}\log\log x. 
        \]
 If $\eta=\delta$, then the extra $\log\log x$  above can be removed. 
\end{proposition}
\begin{proof}
   We first treat the case  $0<\eta<\delta$. Suppose, for contradiction, that there are infinitely many $x\in \R$ satisfying   
    \begin{align*}
G_{\mathcal{G}_\eta}(x)\geq \mathfrak{C}(\Delta, \delta) x^{\delta+\eta}\log\log x,
\end{align*}
where $\mathfrak{C}(\Delta, \delta)$ will be a constant to be specified later.  
   
For each such $x$ and each prime $p$ counted by $G_{\mathcal G_\eta}(x)$,
choose an integer $f_{0}(p)\in\mathcal G_\eta$
such that
\[
p\mid
\operatorname{Num}\!\left(
P_{f_0(p)^2\Delta^*}(j_E)
\right).
\]
Such a choice exists by the definition of
$G_{\mathcal G_\eta}(x)$. Now apply \Cref{prop:Charles} to each $N\in \mathcal{G}_\eta$ and then sum over $N$.  
Using the bound $e_N\ll N\log\log N$,  \eqref{eq:arch-estimation}, and \eqref{eq:non-arch-lower}, if $\eta<\delta$, then the left hand side of \eqref{eq:Charles} becomes
\begin{align*}
& \sum_{N\in \mathcal{G}_\eta}\left( \widehat{\deg}(Z.t_{N*}Y)- \log||s_Z(t_{N*}Y)||\right)\\
\geq  & \sum_{N\in \mathcal{G}_\eta}\left( \sum_{p \text{ counted by $G_{\mathcal{G}_\eta}(x)$}}\deg_p(Z.t_{N*}Y)-\log||s_Z(t_{N*}Y)||\right)\\
\geq  & \sum_{p \text{ counted by $G_{\mathcal{G}_\eta}(x)$}} 
\deg_p(Z.t_{f_0(p)*}Y)- \sum_{N\in \mathcal{G}_\eta} \log||s_Z(t_{N*}Y)||\\
\geq &  G_{\mathcal{G}_\eta}(x)  \log x^{\delta}-  \#\mathcal{G}_\eta\max_{N\in \mathcal{G}_\eta} h(\Delta^*)e_{N}\log (N\Delta^*)\\
\geq   &    \delta G_{\mathcal{G}_\eta}(x)  \log x-C(\Delta)\left(x^{\eta+\delta} (\log\log x)(\log x^\delta +\log \Delta^*)\right)\\
\geq   &    \delta G_{\mathcal{G}_\eta}(x)  \log x-\delta C(\Delta)\log (3\Delta) x^{\eta+\delta} (\log\log x)(\log x),
\end{align*}
where $C(\Delta)=O_E(\Delta^{1/2}\log (3\Delta))$.

On the other hand, by the observation in \cite[p. 426]{MR2017146} and  \eqref{eq:CMdivisor}, we have 
\[
h'_{\widehat{\mathcal{L}}}(Y)=h'_{\widehat{\mathcal{L}}}(y_{\Delta^*})=\frac{12}{h(\Delta^*)}\sum_{1\leq i\leq h(\Delta^*)} \tilde{h}(j_i)=12\tilde{h}(j_1),
\]
where \(\widetilde{h}(j_i)\) denotes the stable Faltings height of a CM elliptic curve with \(j\)-invariant \(j_i\). By   \cite[Th\'eor\`eme III.2.9]{MR1247996} (see also \cite[Theorem 4.26]{MR4353097}),  we have 
\[
\widetilde{h}(j_1)=-\frac{1}{2}\frac{L'(0, \chi_{\Delta^*})}{L(0, \chi_{\Delta^*})}-\frac{1}{4}\log \Delta^*+O(1),
\]
where $\chi_{\Delta^*}$ is the primitive quadratic character associated to the imaginary quadratic field $\Q(\sqrt{-\Delta})$.
By \cite[Proposition 4.1]{MR4474852} and the recent
result on Landau--Siegel zeros \cite{banks2026openaiprooflandausiegelzeros}, we obtain
\[
\left|\frac{L'(0,\chi_{\Delta^*})}
{L(0,\chi_{\Delta^*})}\right|
\ll \log\Delta^*.
\]
Consequently, we obtain the unconditional bound
$
|\widetilde{h}(j_1)|\ll\log(3\Delta).
$

Thus,  summing the right-hand side of \eqref{eq:Charles} over \(N\in\mathcal{G}_\eta\) gives
\[
\ll  \sum_{N\in \mathcal{G}_\eta}h(\Delta^*)  (e_{N}(\log N+\log (3\Delta)))\leq \delta C'(\Delta)x^{\eta+\delta} (\log\log x)(\log x),
\]
where $C'(\Delta)=O(\Delta^{1/2} (\log (3\Delta))^2)$.

Taking $\mathfrak{C}(\Delta, \delta)>C(\Delta)\log (3\Delta)+C'(\Delta)$, and $x$ sufficiently large, we get a contradiction by comparing the upper and lower bounds after summing over $N$ for \eqref{eq:Charles}. Therefore, for all sufficiently large $x$, we have 
\[
G_{\mathcal{G}_\eta}(x)\ll_E \Delta^{1/2} (\log (3\Delta))^2 x^{\eta+\delta}\log\log x.   
\]

Now assume $\eta=\delta$. In this case, one could avoid the extra $\log\log x$ factor by averaging $e_N$ over $N$, 
 rather than bounding it individually. Using  $e_N=\sum_{d\mid N}\frac{\mu(d)^2 N}{d}$, where $\mu(d)$ is the M\"obius function,  we have 
 \begin{align*}
  \sum_{N\in \mathcal{G}_\delta} \log||s_Z(t_{N*}Y)||
\leq &  \tilde{C}_0(\Delta) \sum_{N\le 2 x^\delta} \sum_{d\mid N}\frac{\mu(d)^2 N}{d}\log (N\Delta^*)\\
= & \tilde{C}_0(\Delta) \sum_{d\le 2x^\delta} \frac{\mu(d)^2 }{d} \sum_{\substack{N\leq 2x^\delta\\d\mid N}}N (\log N+\log \Delta^*)\\
\leq &  \tilde{C}_1(\Delta) \sum_{d\le 2x^\delta} \left\lceil\frac{2x^\delta}{d} \right\rceil ^2(\log x^\delta+\log \Delta^*)\\
\leq & \delta \tilde{C}(\Delta) x^{2\delta}\log x,
\end{align*}
where $\tilde{C}(\Delta)=O_E(\Delta^{1/2} (\log (3\Delta))^2)$. 
Consequently, the preceding lower bound becomes
\[
\delta G_{\mathcal{G}_\delta}(x)  \log x-\delta\tilde{C}(\Delta) x^{2\delta}\log x
\]
Using the same argument as above, the right-hand side of \eqref{eq:Charles} is
\[
\leq \delta \tilde{C}'(\Delta)x^{2\delta} \log x,
\]
where $\tilde{C}'(\Delta)=O(\Delta^{1/2} (\log (3\Delta))^2)$. 
The result for $\eta=\delta$ follows by taking $\mathfrak{C}(\Delta, \delta) \geq \tilde{C}(\Delta)+\tilde{C}'(\Delta)$. 
\end{proof}

Finally, applying \eqref{eq:thm2-G}, \eqref{eq:general-bound}, and  \Cref{prop:general-bound-G}, we obtain 
 $
 G_\delta(x)\ll_E \Delta^{1/2}(\log (3\Delta))^2 x^{2\delta} 
 $.
This completes the proof of \Cref{thm:upper-bound}.

 
\newpage

\begin{table}[ht]\label{table}
\centering
\small
\caption{For each $1\leq f\leq 100$, the table gives a prime
$p\leq 10^{15}$ for which $v_p(P_{7f^2}(j_E))$ is maximal, together
with the corresponding maximal valuation. Here $j_E=10976$.}
\begin{tabular}{ccc|ccc|ccc|ccc}
\hline
$f$ & $p$ & $\max v_p$
& $f$ & $p$ & $\max v_p$
& $f$ & $p$ & $\max v_p$
& $f$ & $p$ & $\max v_p$ \\
\hline
1  & 113 & 1  & 26 & 31 & 6  & 51 & 31 & 28 & 76  & 31 & 15 \\
2  & 31  & 1  & 27 & 31 & 18 & 52 & 31 & 10 & 77  & 31 & 28 \\
3  & 31  & 2  & 28 & 31 & 6  & 53 & 31 & 24 & 78  & 31 & 24 \\
4  & 31  & 1  & 29 & 31 & 10 & 54 & 31 & 10 & 79  & 31 & 34 \\
5  & 31  & 2  & 30 & 31 & 10 & 55 & 31 & 24 & 80  & 31 & 20 \\
6  & 31  & 2  & 31 & 223& 2  & 56 & 31 & 17 & 81  & 31 & 40 \\
7  & 31  & 4  & 32 & 31 & 7  & 57 & 31 & 30 & 82  & 31 & 21 \\
8  & 31  & 1  & 33 & 31 & 16 & 58 & 31 & 14 & 83  & 31 & 36 \\
9  & 31  & 4  & 34 & 31 & 10 & 59 & 31 & 24 & 84  & 31 & 24 \\
10 & 31  & 3  & 35 & 31 & 18 & 60 & 31 & 18 & 85  & 31 & 42 \\
11 & 31  & 4  & 36 & 31 & 12 & 61 & 31 & 26 & 86  & 31 & 14 \\
12 & 31  & 2  & 37 & 31 & 16 & 62 & 31 & 1  & 87  & 31 & 44 \\
13 & 31  & 6  & 38 & 31 & 9  & 63 & 31 & 32 & 88  & 31 & 16 \\
14 & 223 & 1  & 39 & 31 & 20 & 64 & 31 & 9  & 89  & 31 & 34 \\
15 & 31  & 10 & 40 & 31 & 9  & 65 & 31 & 34 & 90  & 31 & 28 \\
16 & 31  & 4  & 41 & 31 & 14 & 66 & 31 & 16 & 91  & 31 & 40 \\
17 & 31  & 6  & 42 & 31 & 8  & 67 & 31 & 24 & 92  & 31 & 20 \\
18 & 31  & 4  & 43 & 31 & 18 & 68 & 31 & 14 & 93  & 223& 6  \\
19 & 31  & 8  & 44 & 31 & 8  & 69 & 31 & 40 & 94  & 31 & 16 \\
20 & 31  & 5  & 45 & 31 & 28 & 70 & 31 & 14 & 95  & 31 & 48 \\
21 & 31  & 12 & 46 & 31 & 4  & 71 & 31 & 28 & 96  & 31 & 30 \\
22 & 31  & 4  & 47 & 31 & 20 & 72 & 31 & 24 & 97  & 31 & 42 \\
23 & 31  & 10 & 48 & 31 & 14 & 73 & 31 & 28 & 98  & 31 & 27 \\
24 & 31  & 4  & 49 & 31 & 18 & 74 & 31 & 12 & 99  & 31 & 48 \\
25 & 31  & 14 & 50 & 31 & 9  & 75 & 31 & 46 & 100 & 31 & 25 \\
\hline
\end{tabular}
\label{tabel}
\end{table}

\begin{table}[ht]
\centering
\caption{For each $1\leq f\leq 100$, the table gives a prime $p\leq 10^{15}$ for which
$v_p(P_{15f^2}(j_E))$ is maximal, together with the corresponding
maximal valuation. Here $j_E=10976$.}
\label{tab:max-padic-valuations-15}
\small
\begin{tabular}{ccc|ccc|ccc|ccc}
\hline
$f$ & $p$ & $\max v_p$
& $f$ & $p$ & $\max v_p$
& $f$ & $p$ & $\max v_p$
& $f$ & $p$ & $\max v_p$ \\
\hline
1  & 10369  & 1      & 26 & $\ast$ & $\ast$ & 51 & 223 & 6  & 76  & 223 & 2  \\
2  & 241    & 1      & 27 & $\ast$ & $\ast$ & 52 & 223 & 6  & 77  & 223 & 6  \\
3  & 1471   & 1      & 28 & 479     & 3      & 53 & 223 & 10 & 78  & 223 & 8  \\
4  & 61     & 1      & 29 & 223     & 2      & 54 & 223 & 2  & 79  & 223 & 14 \\
5  & $91606639$ & 1 & 30 & 479     & 2      & 55 & 223 & 8  & 80  & 223 & 6  \\
6  & 753219313718881 & 1 & 31 & 223     & 2      & 56 & 223 & 2  & 81  & 223 & 10 \\
7  & 223    & 2      & 32 & 223     & 2      & 57 & 223 & 10 & 82  & 223 & 2  \\
8  & 223    & 2      & 33 & 223     & 6      & 58 & 223 & 4  & 83  & 223 & 4  \\
9  & 223    & 4      & 34 & 223     & 2      & 59 & 223 & 6  & 84  & 223 & 8  \\
10 & 223    & 2      & 35 & 223     & 6      & 60 & 223 & 4  & 85  & 223 & 6  \\
11 & 223    & 2      & 36 & 223     & 6      & 61 & 479 & 4  & 86  & 223 & 6  \\
12 & 105529 & 1      & 37 & 223     & 2      & 62 & 223 & 2  & 87  & 223 & 12 \\
13 & 223    & 2      & 38 & 223     & 2      & 63 & 223 & 4  & 88  & 223 & 2  \\
14 & 223    & 2      & 39 & 223     & 4      & 64 & 223 & 6  & 89  & 223 & 14 \\
15 & 223    & 2      & 40 & 223     & 6      & 65 & 223 & 4  & 90  & 223 & 4  \\
16 & $\ast$ & $\ast$ & 41 & 223     & 6      & 66 & 223 & 2  & 91  & 223 & 12 \\
17 & 479    & 2      & 42 & $\ast$ & $\ast$ & 67 & 223 & 8  & 92  & 223 & 4  \\
18 & 26111  & 1      & 43 & 223     & 6      & 68 & 223 & 4  & 93  & 223 & 10 \\
19 & 50671  & 1      & 44 & 25439   & 1      & 69 & 223 & 6  & 94  & 223 & 4  \\
20 & 479    & 1      & 45 & 223     & 6      & 70 & 223 & 4  & 95  & 223 & 12 \\
21 & 223    & 2      & 46 & 223     & 2      & 71 & 223 & 10 & 96  & 223 & 4  \\
22 & 223    & 2      & 47 & 223     & 8      & 72 & 479 & 3  & 97  & 223 & 12 \\
23 & 223    & 2      & 48 & 223     & 4      & 73 & 223 & 12 & 98  & 479 & 10 \\
24 & 31391  & 1      & 49 & 223     & 4      & 74 & 223 & 6  & 99  & 479 & 10 \\
25 & $\ast$ & $\ast$ & 50 & 223     & 4      & 75 & 223 & 12 & 100 & 223 & 2  \\
\hline
\end{tabular}
\label{tabel-2}
\end{table}

\newpage

\bibliographystyle{amsplain}
\bibliography{References}

\providecommand{\bysame}{\leavevmode\hbox to3em{\hrulefill}\thinspace}
\providecommand{\MR}{\relax\ifhmode\unskip\space\fi MR }
\providecommand{\MRhref}[2]{%
  \href{http://www.ams.org/mathscinet-getitem?mr=#1}{#2}
}
\providecommand{\href}[2]{#2}
\begin{thebibliography}{10}

\bibitem{MR223335}
Emil Artin and John Tate, \emph{Class {F}ield {T}heory}, W. A. Benjamin, Inc., New York-Amsterdam, 1968. \MR{223335}

\bibitem{aslanyan2023multiplicative}
Vahagn Aslanyan, Sebastian Eterovi{\'c}, and Guy Fowler, \emph{Multiplicative relations among differences of singular moduli}, arXiv preprint arXiv:2308.12244 (2023).

\bibitem{MR2017146}
Pascal Autissier, \emph{Hauteur des correspondances de {H}ecke}, Bull. Soc. Math. France \textbf{131} (2003), no.~3, 421--433. \MR{2017146}

\bibitem{banks2026openaiprooflandausiegelzeros}
William~D. Banks, \emph{On the {OpenAI} proof that {Landau--Siegel} zeros do not exist}, 2026, arXiv:2610.15483.

\bibitem{MR4190395}
Yuri Bilu, Philipp Habegger, and Lars K\"uhne, \emph{No singular modulus is a unit}, Int. Math. Res. Not. IMRN (2020), no.~24, 10005--10041. \MR{4190395}

\bibitem{MR4276350}
Yulin Cai, \emph{Bounding the difference of two singular moduli}, Mosc. J. Comb. Number Theory \textbf{10} (2021), no.~2, 95--110. \MR{4276350}

\bibitem{MR4296371}
Francesco Campagna, \emph{On singular moduli that are {$S$}-units}, Manuscripta Math. \textbf{166} (2021), no.~1--2, 73--90. \MR{4296371}

\bibitem{MR4545213}
\bysame, \emph{Effective bounds on differences of singular moduli that are {$S$}-units}, Math. Proc. Cambridge Philos. Soc. \textbf{174} (2023), no.~2, 415--450. \MR{4545213}

\bibitem{MR4929976}
Francesco Campagna and Gabriel~A. Dill, \emph{Around the support problem for {H}ilbert class polynomials}, Comment. Math. Helv. \textbf{100} (2025), no.~3, 421--462. \MR{4929976}

\bibitem{MR3843371}
Fran\c{c}ois Charles, \emph{Exceptional isogenies between reductions of pairs of elliptic curves}, Duke Math. J. \textbf{167} (2018), no.~11, 2039--2072. \MR{3843371}

\bibitem{MR2058609}
Laurent Clozel and Emmanuel Ullmo, \emph{\'{E}quidistribution des points de {H}ecke}, Contributions to automorphic forms, geometry, and number theory, Johns Hopkins Univ. Press, Baltimore, MD, 2004, pp.~193--254. \MR{2058609}

\bibitem{MR2464027}
Alina~Carmen Cojocaru and Chantal David, \emph{Frobenius fields for elliptic curves}, Amer. J. Math. \textbf{130} (2008), no.~6, 1535--1560. \MR{2464027}

\bibitem{MR2178556}
Alina~Carmen Cojocaru, Etienne Fouvry, and M.~Ram Murty, \emph{The square sieve and the {L}ang-{T}rotter conjecture}, Canad. J. Math. \textbf{57} (2005), no.~6, 1155--1177. \MR{2178556}

\bibitem{MR4586829}
Alina~Carmen Cojocaru and Tian Wang, \emph{Bounds for the distribution of the {F}robenius traces associated to products of non-{CM} elliptic curves}, Canad. J. Math. \textbf{75} (2023), no.~3, 687--712. \MR{4586829}

\bibitem{MR1247996}
Pierre Colmez, \emph{P\'eriodes des vari\'et\'es ab\'eliennes \`a{} multiplication complexe}, Ann. of Math. (2) \textbf{138} (1993), no.~3, 625--683. \MR{1247996}

\bibitem{MR1295951}
Salvador Comalada, \emph{Twists and reduction of an elliptic curve}, J. Number Theory \textbf{49} (1994), no.~1, 45--62. \MR{1295951}

\bibitem{MR1028322}
David~A. Cox, \emph{Primes of the form {$x^2 + ny^2$}: Fermat, class field theory and complex multiplication}, A Wiley-Interscience Publication, John Wiley \& Sons, Inc., New York, 1989. \MR{1028322}

\bibitem{MR5125}
Max Deuring, \emph{Die {T}ypen der {M}ultiplikatorenringe elliptischer {F}unktionenk\"orper}, Abh. Math. Sem. Hansischen Univ. \textbf{14} (1941), 197--272. \MR{5125}

\bibitem{MR4527824}
Daniel Disegni, \emph{{$p$}-adic equidistribution of {CM} points}, Comment. Math. Helv. \textbf{97} (2022), no.~4, 635--668. \MR{4527824}

\bibitem{MR903384}
Noam~D. Elkies, \emph{The existence of infinitely many supersingular primes for every elliptic curve over {${\bf Q}$}}, Invent. Math. \textbf{89} (1987), no.~3, 561--567. \MR{903384}

\bibitem{MR1144318}
\bysame, \emph{Distribution of supersingular primes}, no. 198-200, 1991, Journ\'ees Arithm\'etiques, 1989 (Luminy, 1989), pp.~127--132. \MR{1144318}

\bibitem{MR2636029}
Noam~David Elkies, \emph{Supersingular primes of a given elliptic curve over a number field}, ProQuest LLC, Ann Arbor, MI, 1987, Thesis (Ph.D.)--Harvard University. \MR{2636029}

\bibitem{MR2476572}
Andreas Enge, \emph{The complexity of class polynomial computation via floating point approximations}, Math. Comp. \textbf{78} (2009), no.~266, 1089--1107. \MR{2476572}

\bibitem{MR1382477}
Etienne Fouvry and M.~Ram Murty, \emph{On the distribution of supersingular primes}, Canad. J. Math. \textbf{48} (1996), no.~1, 81--104. \MR{1382477}

\bibitem{MR4371528}
Guy Fowler, \emph{Multiplicative independence of modular functions}, J. Th\'eor. Nombres Bordeaux \textbf{33} (2021), no.~2, 459--509. \MR{4371528}

\bibitem{FoTr2026}
Guy Fowler and Emanuele Tron, \emph{Distinct differences of singular moduli}, manuscripta mathematica \textbf{177} (2026), no.~5, 62.

\bibitem{MR772491}
Benedict~H. Gross and Don~B. Zagier, \emph{On singular moduli}, J. reine angew. Math. \textbf{355} (1985), 191--220. \MR{772491}

\bibitem{MR3404647}
Philipp Habegger, \emph{Singular moduli that are algebraic units}, Algebra Number Theory \textbf{9} (2015), no.~7, 1515--1524. \MR{3404647}

\bibitem{MR4129386}
Sebasti\'an Herrero, Ricardo Menares, and Juan Rivera-Letelier, \emph{{$p$}-adic distribution of {CM} points and {H}ecke orbits {I}: {C}onvergence towards the {G}auss point}, Algebra Number Theory \textbf{14} (2020), no.~5, 1239--1290. \MR{4129386}

\bibitem{MR4713026}
\bysame, \emph{There are at most finitely many singular moduli that are {$S$}-units}, Compos. Math. \textbf{160} (2024), no.~4, 732--770. \MR{4713026}

\bibitem{HeMeLe2021}
Sebastián Herrero, Ricardo Menares, and Juan Rivera-Letelier, \emph{{$p$}-adic distribution of {C}{M} points and {H}ecke orbits. {II}: Linnik equidistribution on the supersingular locus},  (2021), arxiv:2102.04865.

\bibitem{MR1040429}
Masanobu Kaneko, \emph{Supersingular {$j$}-invariants as singular moduli {${\rm mod}\, p$}}, Osaka J. Math. \textbf{26} (1989), no.~4, 849--855. \MR{1040429}

\bibitem{MR2695524}
David~Russell Kohel, \emph{Endomorphism rings of elliptic curves over finite fields}, ProQuest LLC, Ann Arbor, MI, 1996, Thesis (Ph.D.)--University of California, Berkeley. \MR{2695524}

\bibitem{MR890960}
Serge Lang, \emph{Elliptic functions}, second ed., Graduate Texts in Mathematics, vol. 112, Springer-Verlag, New York, 1987, With an appendix by J. Tate. \MR{890960}

\bibitem{MR568299}
Serge Lang and Hale Trotter, \emph{Frobenius distributions in {${\rm GL}\sb{2}$}-extensions}, Lecture Notes in Mathematics, vol. 504, Springer-Verlag, Berlin-New York, 1976, Distribution of Frobenius automorphisms in ${\rm GL}\sb{2}$-extensions of the rational numbers. \MR{568299}

\bibitem{MR3431591}
Kristin Lauter and Bianca Viray, \emph{On singular moduli for arbitrary discriminants}, Int. Math. Res. Not. IMRN (2015), no.~19, 9206--9250. \MR{3431591}

\bibitem{MR4251608}
Yingkun Li, \emph{Singular units and isogenies between {CM} elliptic curves}, Compos. Math. \textbf{157} (2021), no.~5, 1022--1035. \MR{4251608}

\bibitem{lmfdb}
The {LMFDB Collaboration}, \emph{The {L}-functions and modular forms database}, \url{https://www.lmfdb.org}, 2026.

\bibitem{MR736719}
Guy Robin, \emph{Estimation de la fonction de {T}chebychef {$\theta $}\ sur le {$k$}-i\`{e}me nombre premier et grandes valeurs de la fonction {$\omega (n)$}\ nombre de diviseurs premiers de {$n$}}, Acta Arith. \textbf{42} (1983), no.~4, 367--389. \MR{736719}

\bibitem{MR1085266}
Ren\'e Schoof, \emph{The exponents of the groups of points on the reductions of an elliptic curve}, Arithmetic algebraic geometry ({T}exel, 1989), Progr. Math., vol.~89, Birkh\"auser Boston, Boston, MA, 1991, pp.~325--335. \MR{1085266}

\bibitem{MR644559}
Jean-Pierre Serre, \emph{Quelques applications du th\'eor\`eme de densit\'e{} de {C}hebotarev}, Inst. Hautes \'Etudes Sci. Publ. Math. (1981), no.~54, 123--201. \MR{644559}

\bibitem{MR4875529}
Ananth~N. Shankar and Yunqing Tang, \emph{Reductions of abelian varieties and {K}3 surfaces}, J. Number Theory \textbf{270} (2025), 122--166. \MR{4875529}

\bibitem{MR393039}
John Tate, \emph{Algorithm for determining the type of a singular fiber in an elliptic pencil}, Modular functions of one variable, {IV} ({P}roc. {I}nternat. {S}ummer {S}chool, {U}niv. {A}ntwerp, {A}ntwerp, 1972), Lecture Notes in Math., vol. 476, Springer, Berlin-New York, 1975, pp.~33--52. \MR{393039}

\bibitem{MR3848226}
Jesse Thorner and Asif Zaman, \emph{A {C}hebotarev variant of the {B}run-{T}itchmarsh theorem and bounds for the {L}ang-{T}rotter conjectures}, Int. Math. Res. Not. IMRN (2018), no.~16, 4991--5027. \MR{3848226}

\bibitem{MR265369}
William~C. Waterhouse, \emph{Abelian varieties over finite fields}, Ann. Sci. \'Ecole Norm. Sup. (4) \textbf{2} (1969), 521--560. \MR{265369}

\bibitem{MR4474852}
Xunjing Wei, \emph{Smallness of {F}altings heights of {CM} abelian varieties}, J. Number Theory \textbf{242} (2023), 73--101. \MR{4474852}

\bibitem{MR4353097}
Gisbert W\"ustholz and Clemens Fuchs (eds.), \emph{Arithmetic and geometry---ten years in {A}lpbach}, Annals of Mathematics Studies, vol. 202, Princeton University Press, Princeton, NJ, 2019, Lecture notes from the mini-courses presented at the 10th Alpbach Summerschool held in Alpbach, 2016. \MR{4353097}

\bibitem{MR1826411}
Shouwu Zhang, \emph{Heights of {H}eegner points on {S}himura curves}, Ann. of Math. (2) \textbf{153} (2001), no.~1, 27--147. \MR{1826411}

\bibitem{MR3453123}
David Zywina, \emph{Bounds for the {L}ang-{T}rotter conjectures}, S{CHOLAR}---a scientific celebration highlighting open lines of arithmetic research, Contemp. Math., vol. 655, Amer. Math. Soc., Providence, RI, 2015, pp.~235--256. \MR{3453123}

\end{thebibliography}
\end{document}